\documentclass[]{interact}
\usepackage{graphicx} % Required for inserting images
\usepackage[numbers,sort&compress]{natbib}% Citation support using natbib.sty
\bibpunct[, ]{[}{]}{,}{n}{,}{,}% Citation support using natbib.sty
\renewcommand\bibfont{\fontsize{10}{12}\selectfont}% Bibliography support using natbib.sty
\usepackage{amsfonts}
\usepackage{amssymb}
\usepackage{amsmath}
\usepackage{subfigure}
\usepackage{booktabs}
\usepackage{url}
\usepackage{enumitem}

\usepackage{tikz}
\usetikzlibrary{shapes,arrows}
\usetikzlibrary{calc}

\usepackage[english]{babel}

\usepackage{lscape} 
\usepackage{mdframed} % This one is new!!!
\usepackage{adjustbox}
\usepackage[ruled,vlined]{algorithm2e}
\usepackage{todonotes} 
\usepackage{amsthm}
\usepackage[T1]{fontenc}

\theoremstyle{definition}
\newtheorem{theorem}{Theorem}%
\newtheorem{remark}{Remark}%
\newtheorem{lemma}{Lemma}%
\newtheorem{assumption}{Assumption} 

\newtheorem{definition}{Definition}%

\newcommand{\x}{{\boldsymbol x}}
\newcommand{\y}{{\boldsymbol y}}

\newcommand{\bd}{{\boldsymbol d}}
\newcommand{\bg}{{\boldsymbol g}}
\newcommand{\be}{{\boldsymbol e}}
\newcommand{\bp}{{\boldsymbol p}}

\newcommand{\bs}{{\boldsymbol s}}
\newcommand{\bu}{{\boldsymbol u}}

\newcommand{\bxi}{\mbox{\boldmath$\xi$}}

\newcommand{\R}{\mathbb{R}}
\newcommand{\K}{\mathcal{K}}
\newcommand{\inexactLMBM}{{\tt InexactLMBM}}

\newcommand{\norm}[1]{\lVert#1\rVert}

\DeclareMathOperator{\diag}{diag}
\DeclareMathOperator{\tr}{tr}

\begin{document}

%\articletype{ARTICLE TEMPLATE}% Specify the article type or omit as appropriate

\title{Inexact Limited Memory Bundle Method}

\author{
\name{Jenni Lampainen\textsuperscript{a}\thanks{CONTACT Jenni Lampainen. Email: jmlamp@utu.fi}, Kaisa Joki\textsuperscript{a}, Napsu Karmitsa\textsuperscript{b} and Marko M. Mäkelä\textsuperscript{a}}
\affil{\textsuperscript{a}Department of Mathematics and Statistics, University of Turku, FI-20014 Turku, Finland; \textsuperscript{b}Department of Computing, University of Turku, FI-20014 Turku, Finland}
}

\maketitle

\begin{abstract}
Large-scale nonsmooth optimization problems arise in many real-world applications, but obtaining exact function and subgradient values for these problems may be computationally expensive or even infeasible. In many practical settings, only inexact information is available due to measurement or modeling errors, privacy-preserving computations, or stochastic approximations, making inexact optimization methods particularly relevant. In this paper, we propose a novel inexact limited memory bundle method for large-scale nonsmooth nonconvex optimization. The method tolerates noise in both function values and subgradients. We prove the global convergence of the proposed method to an approximate stationary point. Numerical experiments with different levels of noise in function and/or subgradient values show that the method performs well with both exact and noisy data. In particular, the results demonstrate competitiveness in large-scale nonsmooth optimization and highlight the suitability of the method for applications where noise is unavoidable, such as differential privacy in machine learning.
\end{abstract}


\begin{amscode}
90C26, 90C06, 49J52, 65K05
\end{amscode}

\section{Introduction}
\label{sec_Introduction}
Nonsmooth optimization problems, in which the functions are not necessarily continuously differentiable, arise naturally in a wide range of real-world applications. These include engineering \cite{MisSta:1998}, mechanics \cite{MorPanStr:1988}, economics \cite{OutKocZow:1998}, and computational chemistry \cite{YanChe:2004}, as well as image and signal processing \cite{KarMajMak:2001}, to name a few. For instance, in economics nonsmoothness is encountered in equilibrium models, location problems, and other decision-making tasks \cite{BouFadRou:2026, Demyanov:2013, MovNob:2010, SinMis:2026}. Similarly, in image and signal processing, nonsmooth formulations are widely used in reconstruction, denoising, and inverse problems \cite{RamArg:2015, Nik:2010, WanChe:2022}. One of the most visible modern application domains is machine learning and data analysis, where nonsmooth optimization problems arise in tasks such as clustering \cite{BagKarTah:2025, BagTahUgo:2016, KarBagTah:2018, bigclust2025, LamJokKar:2026}, regression \cite{BagUgo:2017, BagTahKarEtal:2022, KarTahJokMakBagMak2023, oscar:2023}, and classification \cite{AstFud:2007, AstFudGor:2008, KarJokAir:2025, pauliina}.

Research in nonsmooth optimization is mainly focused on convex problems. In the convex case, the optimization process can be simplified, and the global optimality of a solution can be guaranteed. However, in practical applications, nonconvex problems are often more common. Furthermore, many modern applications, especially those arising in machine learning, involve very large datasets containing millions of data points and high-dimensional feature vectors, which impose strict requirements on computational efficiency and memory usage. Despite the practical importance of large-scale nonsmooth nonconvex optimization problems, only a few algorithms are capable of efficiently solving them. Among these, the limited memory bundle method (LMBM) \cite{HaaMieMak:2004,HaaMieMak:2007} and its variations (see, e.g., \cite{KarBag:2012, bigclust2025, KarMak:2010, KarJokAir:2025}) are particularly notable.

In practice, exact function and subgradient information are often unavailable. This may result from measurement errors, noisy data, model approximations, reliance on stochastic simulation, or the addition of privacy-preserving noise, for example in the context of differential privacy \cite{Dwork2006, DworkRoth2014, DPraport2025}. While several inexact nonsmooth optimization methods have been proposed, most are restricted either to convex settings \cite{DeoSol:2020, KhaMorTra:2024optmethods, YueZhoSo:2019} or to small- and medium-scale nonconvex problems \cite{HarSagSol:2016, HerUlb:2019, HosNob:2023, LvPanMen:2018, Noll:2013}. To the best of our knowledge, no inexact methods currently exist for nonsmooth nonconvex optimization that can efficiently handle large-scale problems.

In this paper, we present \inexactLMBM, a novel \emph{inexact limited memory bundle method} for large-scale nonsmooth and nonconvex optimization. Specifically, we consider unconstrained optimization problems of the form
\begin{align}
  \label{opt_teht}
  \begin{cases}
    \text{minimize}\quad  &f({\x}) \\
    \text{subject to} &\x \in \R^n,
  \end{cases}
\end{align}
where the objective function $f:\R^n \rightarrow \R$ is locally Lipschitz continuous. The proposed method extends the classical LMBM framework \cite{HaaMieMak:2004, HaaMieMak:2007} to the inexact setting by explicitly incorporating inexact function and subgradient information, while retaining, and potentially improving, the efficiency of the original LMBM. This is achieved by employing a modified (i.e., tilted) subgradient together with a modified subgradient locality measure, as derived from those proposed in \cite{HarSagSol:2016}. These modifications provide sufficient control over the problem in the presence of inexact information and, at the same time, make a line search procedure unnecessary. As a result, the proposed {\inexactLMBM} does not require the objective function to be semi-smooth \cite{Bih:1984}, in contrast to most bundle methods with a line search procedure \cite{BagGauKarMakTah2020} including the original LMBM.

We establish global convergence of the proposed method to an approximate stationary point. Here, global convergence does not refer to convergence to a global minimizer, but rather to the fact that convergence is guaranteed from any starting point. The convergence of {\inexactLMBM} is guaranteed under mild assumptions:
\begin{itemize}
\item the objective function is locally Lipschitz continuous;
\item the level set is bounded for each starting point;
\item the sequence of convexification parameters is bounded; and
\item the errors arising in function and subgradient evaluations remain bounded.
\end{itemize}
It is worth to note that achieving global convergence without assuming semi-smoothness constitutes a significant advancement.

The performance of {\inexactLMBM} is evaluated through numerical experiments. In the case of exact function and subgradient information, the method is compared with relevant benchmark methods, including the original LMBM. To study the effects of inexactness, the algorithm is tested under various scenarios, including noisy subgradients and cases where both function and subgradient evaluations are inexact.

The remainder of this paper is organized as follows. Section \ref{sec_Limited_Memory_Variable_Metric_Bundle_Method} describes \inexactLMBM. In Section \ref{sec_Convergence_Analysis_in_LMBM}, we prove the global convergence of the method. The results of numerical experiments are presented in Section \ref{sec_Numerical_Experiments}. Finally, Section \ref{sec_Conclusion} concludes the paper. All test problems, additional results, and a detailed description of the matrix-updating procedure are provided in the Appendices.

\section{Inexact Limited Memory Bundle Method}
\label{sec_Limited_Memory_Variable_Metric_Bundle_Method}

In this section, we describe the new {\inexactLMBM}. The algorithm retains the main structure of the original LMBM, including the use of serious and null steps, aggregation of subgradients, and limited memory variable metric updates for computing the search direction. However, it differs from the original method in two important aspects: function values and subgradients may be inexact, and the line search procedure used in many nonconvex bundle methods -- including the original LMBM -- can be omitted. The overall structure of the new method is illustrated in the flowchart shown in Figure~\ref{flowchart}.

%-----------FLOWCHART ALKAA-----------

% Define block styles

\tikzstyle{decision1} = [diamond, draw, fill=blue!10, 
text width=7.5em, text badly centered, inner sep=0pt]

\tikzstyle{decision2} = [diamond, draw, fill=blue!10, 
text width=11em, text badly centered, inner sep=0pt]

\tikzstyle{alku} = [rectangle, draw, fill=blue!10,
    text width=8em, text centered, minimum height=3em,line width=1.0mm]
    
\tikzstyle{block1} = [rectangle, draw, fill=blue!10,
    text width=11em, text centered, minimum height=2em]
    
\tikzstyle{block2} = [rectangle, draw, fill=blue!10,
    text width=8em, text centered, minimum height=2.6em]

\tikzstyle{end} = [rectangle, draw, fill=blue!10, line width=0.7mm, text width=4.4em, text centered, minimum height=2.5em]
    
\tikzstyle{line} = [draw, line width=0.45mm, color=black!100, -latex']

\begin{figure}[h!]
\centering
\resizebox{0.85\linewidth}{!}{
\begin{tikzpicture}[node distance = 1.5cm, auto]

%------- Place nodes--------
		
    \node [alku] (keski1) {\textbf{\,Initialization}};

    \node [block2, below of=keski1, node distance=2.2cm] (keskiuusi1) {\textbf{Direction finding}
    };
		
    \node [decision2, below of=keskiuusi1, aspect=2, node distance=3.0cm] (keski2) {\textbf{Desired accuracy?}};
   
    \node [block1, left of=keski2, node distance=5.5cm] (vasen1) {\textbf{Direction finding}\\using the limited\\ memory \textbf{BFGS} update};   
    
    \node [block1, right of=keski2, node distance=5.5cm] (oikea1) {\textbf{Direction finding}\\using the limited\\ memory \textbf{SR1} update};   

    \node [decision1, below of=keski2, node distance=6cm] (keski3) {Does the new \textbf{auxiliary point} improve the solution enough?
    }; 

    \node [block2, right of=keski3, node distance=5.5cm] (oikea2) {\textbf{Null step\\and aggregation}
    };

    \node [block2, left of=keski3, node distance=5.5cm] (vasen2) {\textbf{Serious step}};

    \coordinate (uppermid2) at ($(keski2)!0.5!(oikea1)$);
    \coordinate (lowermid2) at ($(keski3)!0.5!(oikea2)$);
    \node [end] (keskiblokki) at ($(uppermid2)!0.5!(lowermid2)$){\textbf{STOP}};

    \coordinate (uppermid1) at ($(vasen1)!0.5!(keski2)$);
    \coordinate (lowermid1) at ($(vasen2)!0.5!(keski3)$);
    \node [block2] (keskiuusi2) at ($(uppermid1)!0.5!(lowermid1)$) {\textbf{Compute an auxiliary point}
    };

% --- NUOLET ---

% keskimmäinen sarake
\path [line] (keski1) -- (keskiuusi1);
\path [line] (keskiuusi1) -- (keski2);
\path [line] (keski2.south west) -- (keskiuusi2.north) node[midway, right]{No};
\path [line] (keskiuusi2.south) -- (keski3);

% oikea haara
\path [line] (keski3.east) -- (oikea2.west)  node[midway, above]{No};
\path [line] (oikea2.north) -- (oikea1.south);
\path [line] (oikea1.west) -- (keski2.east);

% vasen haara
\path [line] (keski3.west) -- (vasen2.east)  node[midway, above]{Yes};
\path [line] (vasen2.north) -- (vasen1.south);
\path [line] (vasen1.east) -- (keski2.west);

% keskiblokkiin
\path [line] (keski2.south east) -- (keskiblokki.north)  node[midway, right]{Yes};
    
\end{tikzpicture}
}
\caption{Flowchart of {\inexactLMBM}.}
\label{flowchart}
\end{figure}

%-----------FLOWCHART LOPPUU-----------
\vspace{-0.3cm} % ihan vaan että r>0 tuli samalle sivulle

\paragraph*{Notations and preliminaries.}
We begin by introducing some basic notations. Bolded symbols are used to denote vectors, and $\norm{ \, \cdot \,}$ is the Euclidean norm in $\R^n$. The inner product is defined by $\x^\top \y = \sum_{i=1}^n x_i y_i$, and $I \in \R^{n \times n}$ is the \emph{identity matrix}. In addition, we denote by $B_r$ the \emph{open ball} centered at the origin with radius $r>0$.

\texttt{InexactLMBM} generates a sequence of \emph{basic points} $\{\x_k\} \subset \R^n$ together with a sequence of \emph{auxiliary points} $\{\y_k\} \subset \R^n$. At each iteration $k$, the basic point $\x_k$ serves as the stability center and retains the best known solution found so far, whereas at the auxiliary point $\y_k$ new information -- namely an inexact function value and subgradient -- is computed. This information is stored as a \emph{bundle element} and used to steer the solution process when needed. Furthermore, each iteration is classified as either a serious step or a null step: in a serious step, the basic point is updated, whereas in a null step the update is rejected.

Throughout, we assume that the objective function $f$ in problem \eqref{opt_teht} is locally Lipschitz continuous. A function $f: \R^n \to \R$ is \emph{locally Lipschitz continuous} on $\R^n$ if, for any bounded subset $X \subset \R^n$, there exists a constant $L > 0$ such that
$$|f(\x)-f(\y)| \leq L \norm{\x - \y} \quad \text{for all} \,\,\, \x, \y \in X.$$

The \emph{Clarke subdifferential} of a locally Lipschitz continuous function $f: \R^n \to \R$ at any point $\x \in \R^n$ is defined as \cite{Cla:1983}
$$
\partial f(\x) = \operatorname{conv} \left\{ \lim_{i \to \infty} \nabla f(\x_i) \;|\; \x_i \to \x \; \text{ and } \; \nabla f(\x_i) \text{ exists} \right\},
$$
where '$\operatorname{conv}$' denotes the convex hull of a set. Each $\boldsymbol{\xi} \in \partial f(\x)$ is called a \emph{subgradient} of $f$ at $\x$. Since the new method relies on inexact information, we define the following function and subgradient values using inexact oracles:
\begin{itemize}
    \item $f_k = f(\y_k ) - q_k$,\, where $q_k$ is an unknown error; and
    \item $\bxi_k \in \partial f(\y_k ) + B_{r_k}$,\, where $r_k$ is an unknown error.
\end{itemize}
The sign of the error $q_k$ is not fixed, and therefore the true function value can be either overestimated or underestimated. The error terms $q_k$ and $r_k$ are assumed to be bounded. This means that there exist nonnegative error bounds $\bar{q}$ and $\bar{r}$ such that
\begin{align*}
    |q_k| \leq \bar{q}\quad \text{and}\quad 0\leq r_k \leq \bar{r}\quad \text{for all } k.
\end{align*}
Note that if $\bar{q}=\bar{r}=0$, the used function and subgradient values are exact. Moreover, it always holds that $\x_k = \y_m$, where $m$ is the index of the iteration after the latest serious step. In what follows, we denote the noisy objective function value at $\x_k$ by $\hat{f}_k$. In other words, we have
$$
\hat{f}_k = f_m,\quad \text{for all } k\geq m.
$$

\paragraph*{Bundle elements.}
As already mentioned, we compute a new bundle element at each auxiliary point $\y_k$. This element consists not only of the inexact objective function value and subgradient at $\y_k$, but also of a locality measure that quantifies how well the bundle element approximates the objective function value at the current basic point $\x_k$. To be more specific, we apply a modified locality measure and a modified subgradient inspired by \cite{HarSagSol:2016}. The bundle element at $\y_k$ is defined as the triplet 
$$(\y_k, \bxi_k^{mod}, \beta_k),$$ 
where 
$$\bxi^{mod}_{k}=\bxi_{k}+\eta_k(\y_{k}-\x_k)$$
is the \emph{modified subgradient} (i.e., modified slope) and 
$$
\beta_k= \alpha_k + \frac{\eta_k}{2}\norm{\y_k-\x_k}^2,
$$
is the \emph{modified locality measure}. In addition, 
\begin{align}
\label{kaava_alpha}
    \alpha_{k}=\hat{f}_k - f_k -\bxi_{k}^\top(\x_k-\y_{k})
\end{align}
is the \textit{linearization error} and the \textit{convexification parameter} $\eta_k$ is defined by
\begin{align}
\label{kaava_eta}
    \eta_k = \begin{cases}
        \max \left\{ \frac{-2\alpha_k}{\norm{\y_{k}-\x_k}^2},0\right\}+\gamma, \quad &\x_k \neq \y_k \\
        \gamma, \quad &\text{otherwise},
    \end{cases}
\end{align}
where $\gamma > 0$ is a small scalar. It is easy to see, that $\eta_k \geq \gamma > 0$ always holds.

The modified subgradient and locality measure ensure that the line search is no longer required in the new method: whenever necessary, a change in the next search direction can be guaranteed by using the bundle element computed at the new auxiliary point with a constant stepsize $t_k=1$ (or, more generally, any stepsize $t_k \in (0,1]$). This property does not hold for the original LMBM without employing a specific two-step line search procedure \cite{VlcLuk:2001} (see also \cite{HaaMieMak:2004, HaaMieMak:2007}), which in turn relies on an additional semi-smoothness assumption \cite{Bih:1984} not needed here. Note, that if $\x_k = \y_k$, then $\bxi^{mod}_k = \bxi_k$ and the bundle element is $(\x_k, \bxi_k, 0)$. In addition, the following lemma ensures that $\beta_k \geq 0$ holds.

\begin{lemma}
\label{lemmabeta}
    For the modified locality measure, we always have 
    \begin{align}
    \label{kaavabeta}
        \beta_k \geq \frac{\gamma}{2} \norm{\y_k-\x_k}^2.
    \end{align}
    In addition, $\beta_k = 0$ only when $\x_k = \y_k$.
\end{lemma}
\begin{proof}
    We first show that the inequality \eqref{kaavabeta} holds when $\x_k \neq \y_k$. We divide the analysis into two parts based on the sign of the linearization error. 
    If $\alpha_k \geq 0$ then $\eta_k=\gamma > 0$, and
    $$\beta_k=\alpha_k + \frac{\eta_k}{2}\norm{\y_k-\x_k}^2 \geq \frac{\gamma}{2}\norm{\y_k-\x_k}^2.$$
    If $\alpha_k < 0$, then $\eta_k = \frac{-2\alpha_k}{\norm{\y_k-\x_k}^2}+\gamma > 0$, and
  $$\beta_k=\alpha_k + \frac{\eta_k}{2}\norm{\y_k-\x_k}^2 = \alpha_k + \frac{1}{2} \left( \frac{-2\alpha_k}{\norm{\y_k-\x_k}^2}+\gamma \right) \norm{\y_k-\x_k}^2 = \frac{\gamma}{2}\norm{\y_k-\x_k}^2.$$ Therefore, \eqref{kaavabeta} holds when $\x_k \neq \y_k$.

  Finally, if $\x_k=\y_k$ then $\alpha_k = 0$ and due to the definition of the modified locality measure, it is trivial to see that also $\beta_k=0$. Furthermore, in this case $\beta_k \geq \frac{\gamma}{2}\norm{\y_k-\x_k}^2 =0$. Thus, the inequality \eqref{kaavabeta} always holds and also shows that whenever $\x_k \neq \y_k$ then $\beta_k > 0$. This completes the proof.
\end{proof}

\paragraph*{Search direction and stepsize.} Next, we describe the main steps of the new method. First, at iteration $k$, a search direction is generated as
$$\bd_k=-D_k \tilde{\bxi}_k,$$
where $\tilde{\bxi}_k$ denotes an (aggregated) subgradient obtained from the current bundle, and $D_k$ is a variable metric matrix that approximates the inverse Hessian in the smooth case. After the search direction $\bd_k$ is generated, we determine a stepsize $t_k \in [t_{min}, 1]$, where $t_{min} \in (0,1]$ is a positive lower bound. In this process, a sufficient number of the most recent bundle elements -- including their function values and modified subgradients -- are used to estimate a suitable stepsize. Note, that the bundle elements themselves are not updated after being computed. Consequently, the stepsize calculation is computationally inexpensive. The procedure can be regarded as a heuristic step, which often improves practical performance, but the method remains theoretically valid even without it. This heuristic step corresponds to the initial stepsize determination used in the original LMBM and its predecessor, the variable metric bundle method introduced in \cite{VlcLuk:2001}. However, in these methods, the stepsize determination must be continued by an additional line search. In our new method, the use of modified subgradients and differently defined locality measures provide sufficient control over the problem, making the line search unnecessary. For further details on stepsize selection in \inexactLMBM, see the description of the initial stepsize in \cite{VlcLuk:2001}.

\paragraph*{Serious and null steps.} 
When the search direction $\bd_k$ and the stepsize $t_k$ are determined, we define the new auxiliary point by
$$\y_{k+1}=\x_k + t_k \bd_k.$$
Then we evaluate the inexact function value and subgradient and test a decrease condition. The \emph{sufficient descent criterion} is defined by
\begin{align}
\label{eq_serious_descent_criterion_in_LMBM}
f_{k+1} - \hat{f}_k \leq - \varepsilon_L t_k w_k,
\end{align}
where $\varepsilon_L \in (0, 1/2)$ and $w_k >0$ represents the desirable amount of descent of $f$ at $\x_k$. If the condition \eqref{eq_serious_descent_criterion_in_LMBM} is satisfied, a \textit{serious step} is taken: we set $\x_{k+1}=\y_{k+1}$ and $\hat{f}_{k+1}=f_{k+1}$ and add the element $(\x_{k+1}, \bxi_{k+1}, 0)$ to the bundle. Otherwise, if the condition \eqref{eq_serious_descent_criterion_in_LMBM} is not satisfied, a \textit{null step} is taken: we set $\x_{k+1}=\x_k$ and $\hat{f}_{k+1}=\hat{f}_k$ and add the new element $(\y_{k+1}, \bxi_{k+1}^{mod}, \beta_{k+1})$ corresponding to $\y_{k+1}$ to the bundle, while the basic point remains unchanged. For the bundle element corresponding to the null step, the following lemma holds.

\begin{lemma}
\label{lemma_nullstep}
If a null step is performed during iteration $k$, then the new bundle element $(\y_{k+1}, \bxi_{k+1}^{mod}, \beta_{k+1})$ satisfies the property
\begin{align}
\label{cond_null_step_in_LMBMinexact}
- \beta_{k+1} + t_k \bd_k^{\top} \bxi^{mod}_{k+1} \geq f_{k+1} - \hat{f}_k > -\varepsilon_L t_k w_k.
\end{align}
\end{lemma}
\begin{proof}
By definition, $$\bxi^{mod}_{k+1}=\bxi_{k+1}+\eta_k t_k \bd_k$$
and
$$
\beta_{k+1} = \alpha_{k+1} + \frac{\eta_k}{2} \|{t_k \bd_k}\|^2=\hat{f}_k - f_{k+1} + t_k \bd_k^{\top}\bxi_{k+1} + \frac{\eta_k}{2} \|{t_k \bd_k}\|^2,
$$
because in a null step $\hat{f}_{k+1} = \hat{f}_k$. Therefore,
\begin{align*}
- \beta_{k+1} + t_k \bd_k^{\top} \bxi^{mod}_{k+1} & = -\hat{f}_k + f_{k+1} - t_k \bd_k^{\top}\bxi_{k+1} - \frac{\eta_k}{2} \|{t_k \bd_k}\|^2 + t_k \bd_k^{\top} \bxi^{mod}_{k+1}\\
& = -\hat{f}_k + f_{k+1} - t_k \bd_k^{\top}\bxi_{k+1} - \frac{\eta_k}{2} \|{t_k \bd_k}\|^2 + t_k \bd_k^{\top} (\bxi_{k+1}+\eta_k t_k \bd_k) \\ 
& = -\hat{f}_k + f_{k+1} - t_k \bd_k^{\top}\bxi_{k+1} - \frac{\eta_k}{2} \|{t_k \bd_k}\|^2 + t_k \bd_k^{\top} \bxi_{k+1}+\eta_k\|t_k \bd_k\|^2 \\
& = -\hat{f}_k + f_{k+1} + \frac{\eta_k}{2} \|{t_k \bd_k}\|^2 \\
& \geq f_{k+1}-\hat{f}_k.
\end{align*}
Since this is a null step, we have
$$f_{k+1}-\hat{f}_k > -\varepsilon_L t_k w_k.$$
This completes the proof.
\end{proof}

\noindent As shown in Lemma \ref{lemma_nullstep}, the way we define the modified subgradient and locality measure ensures that the inequality \eqref{cond_null_step_in_LMBMinexact} is always satisfied at a null step, even if the linearization error $\alpha_k$ is negative. This property is essential for the convergence analysis and is not guaranteed by the original LMBM without the line search.

\paragraph*{Aggregation.}
The \emph{aggregation procedure} in {\inexactLMBM} is similar to that of the original LMBM and in it we calculate updated values of the aggregate subgradient and the aggregate locality measure. It uses only two bundle elements together with the current aggregate subgradient and locality measure to compute updated aggregate values. Consequently, {\inexactLMBM} involves three subgradients and two locality measures in total. By contrast, standard bundle methods typically employ $n+3$ bundle elements, which significantly increases the computational burden in large-scale problems \cite{KarBagMak:2012}. If more than two bundle elements are stored in the proposed method, they are used solely for determining the stepsize $t_k$. The main difference compared with the original LMBM is that we incorporate modified subgradients rather than standard ones. As previously, let $m$ denote the index of the iteration after the latest serious step. Suppose we have available pairs $(\bxi_m,0)$ and $(\bxi_{k+1}^{mod},\beta_{k+1})$, evaluated at $\x_m$ and $\y_{k+1}$, respectively, along with the current aggregate subgradient $\tilde{\bxi}_k$ and locality measure $\tilde{\beta}_k$ (with $\tilde{\bxi}_1=\bxi_1$ and $\tilde{\beta}_1=0$) . The new aggregate values $\tilde{\bxi}_{k+1}$ and $\tilde{\beta}_{k+1}$ are then defined as a convex combination
\begin{align*}
  \tilde{\bxi}_{k+1} = \lambda^k_{1} \bxi_m + \lambda^k_{2} \bxi_{k+1}^{mod} +
  \lambda^k_{3} \tilde{\bxi}_k \quad \text{and} \quad
  \tilde{\beta}_{k+1} = \lambda^k_{2} \beta_{k+1} + \lambda^k_{3} \tilde{\beta}_k,
\end{align*}
where the coefficients $\lambda^k_i$ satisfy $\lambda^k_i \geq 0$ for all $i \in \,\{1,2,3\,\}$ and $\sum_{i=1}^{3} \lambda^k_i=1$. These coefficients can be obtained by minimizing a straightforward quadratic function that depends on the three subgradients and the two locality measures (see Step~8 in Algorithm~\ref{Alg_LMBM_in_LMBM}). Importantly, this aggregation is performed only when a null step occurs at iteration $k$. In the case of a serious step, $\x_{k+1}=\y_{k+1}$ and we simply set
$$
\tilde{\bxi}_{k+1} = \bxi_{k+1} \in \partial f(\y_{k+1})+ B_{r_{k+1}} \quad \text{and} \quad \tilde{\beta}_{k+1}=0.
$$

\paragraph*{Matrix updating.} The matrix $D_k$ is not formed explicitly. Instead, the search direction $\bd_k$ is computed using a limited memory approach \cite{ByrNocSch:1994} (see also \cite{HaaMieMak:2004, HaaMieMak:2007}). The basic idea is that, rather than storing the matrices $D_k$, information from a small number (say $\hat{m}_c$) of recent iterations is used to implicitly define the variable metric matrix. More precisely, at each iteration the \emph{correction vectors} $\bs_k$ and  $\bu_k$ are stored, and the $\hat{m}_c$ most recent ones are used to define $D_k$. The correction vectors are defined by
$$ \bs_k = \y_{k+1}-\x_k \quad \text{and} \quad \bu_k = \bxi_{k+1}^{mod}-\bxi_m, $$
where $\y_{k+1}$ denotes the newest auxiliary point, $\bxi_{k+1}^{mod}$ the newest modified subgradient and $\bxi_m$ the subgradient calculated at the latest serious step. The correction vectors are used when performing limited memory updates: the limited memory BFGS (L-BFGS) updates after serious steps and the limited memory SR1 (L-SR1) updates after null steps.

It is worth noting, that the condition
\begin{align}
    \label{cond_SR1_pos_def_in_LMBM}
    - \bd_k^\top \bu_k - \tilde{\bxi}_k^\top \bs_k < 0
\end{align}
is checked before updating $D_k$ to $D_{k+1}$, and the update is simply skipped (i.e., we set $D_{k+1} = D_k$) if condition \eqref{cond_SR1_pos_def_in_LMBM} is not satisfied. This directly guarantees that $D_{k+1}$ is positive definite whenever it is obtained by the L-SR1 update \cite{HaaMieMak:2007}. Moreover, in \cite{HaaMieMak:2007} it is shown that condition \eqref{cond_SR1_pos_def_in_LMBM} implies that $\bu_k^\top \bs_k > 0$, which in turn ensures the positive definiteness of the matrices obtained by the L-BFGS update (see, e.g., \cite{ByrNocSch:1994}). Therefore, all matrices $D_{k+1}$ used in defining the search direction are positive definite. A more detailed description of the matrix-updating procedure is given in Appendix~\ref{appendix_matrixupdating}.

\paragraph*{Algorithm.} We present {\inexactLMBM } as Algorithm  \ref{Alg_LMBM_in_LMBM}.

\begin{algorithm}[!htbp]
{
\caption{\inexactLMBM}
\label{Alg_LMBM_in_LMBM}
{\small
\begin{minipage}{0.96\textwidth}
\begin{description}%[leftmargin=-0.3cm]
  \item[Data:] Select the final accuracy tolerance $\varepsilon > 0$, the parameter $\varepsilon_L \in (0,1/2)$, the tolerance $\gamma > 0$, the lower bound $t_{min} \in (0,1]$, the control parameter $C>0$ for the length of the direction vector, and the correction parameter $\varrho \in (0,1/2)$.

  \smallskip
  \item[Step 0:] ({\em Initialization.}) Choose a starting point $\x_1 \in \R^n$ and set $\y_1 \leftarrow \x_1$. Calculate $f_1$ and $\bxi_1$. Set $\hat{f}_1 \leftarrow f_1$, $\beta_1 \leftarrow 0$, and $\bxi_1^{mod} \leftarrow \bxi_1$. Set the correction indicator $i_C \leftarrow 0$, an initial matrix $D_1 \leftarrow I$, and the iteration counter $k \leftarrow  1$.
  
  \smallskip
  \item[{Step 1:}] ({\em Serious step initialization.}) Set the aggregate subgradient $\tilde{\bxi}_k \leftarrow \bxi_k$ and the aggregate locality measure $\tilde{\beta}_k \leftarrow 0$. Set the correction indicator $i_{CN} \leftarrow 0$ for consecutive null steps and the serious step index $m \leftarrow k$.

  \smallskip
  \item[{Step 2:}] ({\em Direction finding.}) Compute
  \begin{align}
    \label{eq__d_k_in_LMBM} 
    \bd_k \leftarrow - D_k \tilde{\bxi}_k
  \end{align}
  by using a L-BFGS update if $m=k$ and by using a L-SR1 update, otherwise. Note that for $k=1$ we set $\bd_1 \leftarrow - \bxi_1$.

  \smallskip
  \item[{Step 3:}] ({\em Correction.})  If $- \tilde{\bxi}_k^{\top}
  \bd_k < \varrho \tilde{\bxi}_k^{\top} \tilde{\bxi}_k$ or $i_{CN}=1$, then
  set
  \begin{align}
    \label{eq_d_bound2_in_LMBM}
    &\bd_k \leftarrow \bd_k - \varrho \tilde{\bxi}_k,
  \end{align}
  (i.e., $D_k \leftarrow D_k+\varrho I$) and $i_C \leftarrow1$. Otherwise, set $i_C \leftarrow0$. If
  $i_C=1$ and $m < k$, then set $i_{CN} \leftarrow1$.

  \smallskip
  \item[{Step 4:}] ({\em Stopping criterion.})  Set
  \begin{align}
    \label{eq_w_k_in_LMBM}
    &w_k \leftarrow - \tilde{\bxi}_k^{\top} \bd_k + 2 \tilde{\beta}_k.
  \end{align}
  If $w_k < \varepsilon$, then {\bf stop} with
  $\x_k$ as the final solution.

  \smallskip
  \item[{Step 5:}] ({\em Auxiliary point.}) Using previously computed bundle elements, calculate the stepsize $t_k \in [t_{min},1]$. If $\norm{\bd_k}>C$, then set the scaled direction vector as $\bd_k \leftarrow \frac{C}{\norm{\bd_k}}\bd_k$. Set $\y_{k+1} \leftarrow \x_k+t_k \bd_k$. Calculate $f_{k+1}$ and $\bxi_{k+1}$. Set $\bs_k \leftarrow \y_{k+1} - \x_k =t_k \bd_k$.

  \smallskip
  \item[{Step 6:}] ({\em Serious step.})
  If $f_{k+1} - \hat{f}_k \leq -\varepsilon_L t_k w_k$, take a serious step: set $\x_{k+1} \leftarrow\y_{k+1}$, $\beta_{k+1} \leftarrow 0$, $\bu_k \leftarrow \bxi_{k+1} - \bxi_m$, $\hat{f}_{k+1} \leftarrow f_{k+1}$ and $k \leftarrow k+1$, and go to Step 1.
  
  \smallskip
  \item[{Step 7:}] ({\em Null step.})
  Take a null step: set $\x_{k+1} \leftarrow \x_k$ and $\hat{f}_{k+1} \leftarrow \hat{f}_k$. Calculate $\alpha_{k+1}$ and $\eta_{k+1}$ using \eqref{kaava_alpha} and \eqref{kaava_eta}, respectively. Set $\bxi_{k+1}^{mod} \leftarrow \bxi_{k+1}+\eta_{k+1} \bs_k$, $\beta_{k+1} \leftarrow \alpha_{k+1} + \frac{\eta_{k+1}}{2} \|{\bs_k}\|^2$ and $\bu_k \leftarrow \bxi_{k+1}^{mod} - \bxi_m$.

  \smallskip
  \item[{Step 8:}] ({\em Aggregation.}) Determine multipliers
  $\lambda^k_{i} \geq 0$ for all $i \in \{1,2,3\}$, $\sum_{i=1}^3
  \lambda^k_{i}=1$ that minimize the strictly convex function
  \begin{align} 
    \label{eq_quad_prob_in_LMBMinexact}
    \varphi (\lambda_1, \lambda_2, \lambda_3) &=\,\, (\lambda_1 \bxi_m
    + \lambda_2 \bxi_{k+1}^{mod} + \lambda_3 \tilde{\bxi}_k)^{\top} D_k
    (\lambda_1 \bxi_m + \lambda_2 \bxi_{k+1}^{mod} + \lambda_3
    \tilde{\bxi}_k) \\ &\qquad + 2(\lambda_2 \beta_{k+1} +
    \lambda_3 \tilde{\beta}_k)\nonumber,
  \end{align}
  where $D_k$ is calculated by the same updating formula as in Step 2
  and $D_k \leftarrow D_k+\varrho I$ if $i_C=1$.  Set
  \begin{align}
    \label{eq_tilde_xi_in_LMBMinexact}
    \tilde{\bxi}_{k+1} &\leftarrow \lambda^k_{1} \bxi_m + \lambda^k_{2}
    \bxi_{k+1}^{mod} + \lambda^k_{3} \tilde{\bxi}_k \qquad \text{and} \\
    \label{eq_tilde_beta_in_LMBMinexact}
    \tilde{\beta}_{k+1} &\leftarrow \lambda^k_{2} \beta_{k+1} + \lambda^k_{3}
    \tilde{\beta}_k.
  \end{align}
  Set $k \leftarrow k+1$ and go to Step~2.

\end{description}
\end{minipage}
}}
\end{algorithm}

%\begin{remark}
%    In Algorithm \ref{Alg_LMBM_in_LMBM}, after a serious step is taken, we have $\x_k=\y_k$. Since the modified subgradient is computed as $\bxi^{mod}_k=\bxi_k + \eta_k (\y_k-\x_k)$, it follows that $\bxi^{mod}_k=\bxi_k$ at Step 1 of Algorithm \ref{Alg_LMBM_in_LMBM}.
%\end{remark}

%\begin{remark}
%    In Step~8 of Algorithm \ref{Alg_LMBM_in_LMBM}, we have $\bxi_m=\bxi^{mod}_m$ because $\bxi_m$ is the subgradient computed after the last serious step and it remains unchanged during consecutive null steps.
%\end{remark}

\begin{remark}
    To ensure convergence of the method, it is necessary that both the length of the direction vector (see Step~5 in Algorithm~\ref{Alg_LMBM_in_LMBM}) and the matrices $B_i = D_i^{-1}$ for all $i=1,\ldots,k$ (see Step~3 in Algorithm~\ref{Alg_LMBM_in_LMBM}) remain bounded. A matrix is said to be bounded if all its eigenvalues belong to a compact interval that excludes zero. 
    Furthermore, the correction~(\ref{eq_d_bound2_in_LMBM}) can be viewed as adding a positive definite matrix $\varrho I$ to $D_k$, which shifts its eigenvalues away from zero. Hence, this correction regularizes $D_k$ and helps preserve the boundedness of its inverse $B_k$.
\end{remark}

\section{Convergence Analysis}
\label{sec_Convergence_Analysis_in_LMBM}

In this section, we analyze the global convergence properties of {\inexactLMBM} under the following assumptions.

\begin{assumption}
\label{assumption_LLC}
The objective function $f : \R^n \to \R$ is locally Lipschitz continuous.
\end{assumption}

\begin{assumption}
\label{assumption_levelset}
The level set $\mathcal{F}(\x_1)=\{\,\x \in \R^n \mid f(\x)\leq f(\x_1) \,\}$ is bounded. 
\end{assumption}

\begin{assumption}
\label{assumption_eta}
The sequence $\{\eta_k\}$ is bounded.
\end{assumption}

\begin{assumption}
\label{assumption_errors}
The errors $q_k$ and $r_k$ in function and subgradient values, respectively, are bounded meaning that there exist nonnegative error bounds $\bar{q}$ and $\bar{r}$ such that $|q_k| \leq \bar{q}$ and $0\leq r_k \leq \bar{r}$ for all $k$.
\end{assumption}

For a locally Lipschitz continuous objective function, a necessary condition for a local minimum $\x^*$ in the unconstrained case is that $\mathbf{0} \in \partial f(\x^*)$. We call $\x^*$ a \emph{stationary point} (see, e.g., \cite{Cla:1983}). In the inexact setting considered here, exact stationarity cannot in general be guaranteed. Therefore, we study convergence to approximate stationary points.

\begin{definition}
A point $\x_k$ is \emph{approximately stationary} for the function $f$ if $\mathbf{0} \in \partial f(\x_k)+ B_{\bar{r}}$, where $\bar{r}$ is the error bound given in Assumption \ref{assumption_errors}.
\end{definition}

Under Assumptions \ref{assumption_LLC} -- \ref{assumption_errors}, we prove that Algorithm~\ref{Alg_LMBM_in_LMBM} either terminates at an approximate stationary point or generates an infinite sequence $\{\x_k\}$ whose accumulation points satisfy the approximate stationarity condition. In particular, we examine the algorithm in the case $\varepsilon = 0$. The convergence analysis follows the same structure as that of the original LMBM \cite{HaaMieMak:2007}, but the results and their proofs are reformulated and extended to accommodate the inexact setting. Whenever a lemma and its proof coincide with those of the original LMBM, we state only the lemma and omit the proof. Modified lemmas and theorems are stated and proved in full detail.

\begin{lemma} 
\label{lemma_k_th_iteration_in_LMBM}
At the $k$th iteration of Algorithm~\ref{Alg_LMBM_in_LMBM}, we have
\begin{alignat}{3}
\label{lemma_3results}
  &w_k = \tilde{\bxi}_k^{\top} D_k \tilde{\bxi}_k + 2\tilde{\beta}_k, \qquad & &w_k
  \geq 2\tilde{\beta}_k, \qquad &
  &w_k \geq \varrho \norm{\tilde{\bxi}_k}^2,
\end{alignat}
and $\tilde{\beta}_k \geq 0$.
Furthermore, if condition \eqref{cond_SR1_pos_def_in_LMBM} is valid, then
\begin{align}
  \label{uDu-us_ge_0_inexact}
  \bu_k^{\top}(D_k\bu_k-\bs_k)>0.
\end{align}
\end{lemma}
\begin{proof} 
  First, we can easily deduce that $\beta_k \geq 0$ for all $k$ based on Lemma \ref{lemmabeta}. Using this fact together with relation (\ref{eq_tilde_beta_in_LMBMinexact}) and Step~1 of
  Algorithm~\ref{Alg_LMBM_in_LMBM}, we conclude that $\tilde{\beta}_k \geq 0$ for all $k$. The relations \eqref{lemma_3results}
  follow directly from (\ref{eq__d_k_in_LMBM})--(\ref{eq_w_k_in_LMBM}). If the
  correction~(\ref{eq_d_bound2_in_LMBM}) is applied, the matrix $D_k$ is replaced by $D_k + \varrho I$. Therefore, the relations \eqref{lemma_3results} remain valid in that case as well.

  It remains to verify that condition~(\ref{cond_SR1_pos_def_in_LMBM}) implies (\ref{uDu-us_ge_0_inexact}). The proof is analogous to the corresponding argument in \cite{HaaMieMak:2007}, with $t_R^k \theta_k$ replaced by $t_k$ and is therefore omitted.
  
%{\color{blue}
%  Next, we prove that condition~(\ref{cond_SR1_pos_def_in_LMBM}) implies (\ref{uDu-us_ge_0_inexact}). Suppose that~(\ref{cond_SR1_pos_def_in_LMBM}) holds. In this case $\tilde{\bxi}_k \ne {\pmb 0}$, since otherwise the equality $-\bd_k^{\top} \bu_k - \tilde{\bxi}_k^{\top} \bs_k = 0$ would hold. Moreover, we obtain
%  \begin{align}
%    \label{cond_SR1_pos_def2_in_LMBMinexact}
%    \bd_k^{\top} \bu_k > -\tilde{\bxi}_k^{\top} \bs_k = t_k \tilde{\bxi}_k^{\top}
%    D_k \tilde{\bxi}_k,
%  \end{align}
%  where $t_k \in [t_{min},1]$ and $t_{min} \in (0,1]$. In addition $t_k \tilde{\bxi}_k^{\top} D_k \tilde{\bxi}_k > 0$, because $D_k$ is positive definite.
  
%  Now we know that
%  $$
%  \bd_k^{\top} \bu_k > -\tilde{\bxi}_k^{\top} \bs_k > 0,
%  $$
%  which implies that the scalar product $\bu_k^{\top}\bs_k$ is positive. Using this fact together with Cauchy's inequality and the relation $\bs_k=t_k \bd_k$, we obtain
%  \begin{align*}
%    (\bu_k^{\top}\bs_k)^2 &= (t_k \tilde{\bxi}_k^{\top} D_k \bu_k)^2 \\
%    &\leq t_k^2\tilde{\bxi}_k^{\top} D_k \tilde{\bxi}_k \bu_k^{\top}
%    D_k \bu_k \\
%    &= t_k \bu_k^{\top} D_k \bu_k (-\bs_k^{\top} \tilde{\bxi}_k).
%  \end{align*}  
%  Applying the inequality (\ref{cond_SR1_pos_def2_in_LMBMinexact}) then gives
%  $$
%  (\bu_k^{\top}\bs_k)^2 < t_k \bu_k^{\top} D_k \bu_k \bd_k^{\top} \bu_k = \bu_k^{\top} D_k \bu_k \bu_k^{\top} \bs_k.
%  $$
%  Dividing by the positive quantity $\bu_k^{\top} \bs_k$ yields $\bu_k^{\top} \bs_k< \bu_k^{\top} D_k \bu_k$, which completes the proof.
%} %end of blue text
\end{proof}

\begin{lemma} 
\label{lemma_lambda_convex_combination_in_LMBM}
Suppose that Algorithm~\ref{Alg_LMBM_in_LMBM} is not terminated before the $k$th iteration and let $m \leq k$ be an index after the latest serious step. Then, there exist numbers $\lambda^{k,j} \geq 0$ for $j =m,\ldots,k$ and $\tilde{\sigma}_k \geq 0$ such that
\begin{align*}
  (\tilde{\bxi}_k,\tilde{\sigma}_k) = \sum_{j=m}^k \lambda^{k,j}
  (\bxi^{mod}_j,\norm{\y_j-\x_j}),\qquad \sum_{j=m}^k \lambda^{k,j} = 1, \quad\text{and} \quad
  \tilde{\beta}_k \geq \frac{\gamma}{2} \tilde{\sigma}_k^2.
\end{align*}
Note, that $\x_j=\x_m$ for $j=m, \ldots, k$.
\end{lemma}
\begin{proof}
   By the assumption, $m$ corresponds to an iteration index following the most recent serious step defined at Step~1 of Algorithm~\ref{Alg_LMBM_in_LMBM}, meaning that $\x_j=\x_m$ for all $j=m,\ldots,k$. First, we prove the existence of nonnegative coefficients $\lambda^{k,j} \geq 0$ for $j = m,\ldots,k$, such that
  \begin{align}
    \label{eq_convex_combination_in_inexactLMBM}
    (\tilde{\bxi}_k,\tilde{\beta}_k) = \sum_{j=m}^k
    \lambda^{k,j}(\bxi^{mod}_j,\beta_j), \qquad \sum_{j=m}^k \lambda^{k,j} = 1.
  \end{align}
  We proceed by induction. For the base case $k=m$, we simply take $\lambda^{m,m}=1$, since at Step~1 of Algorithm~\ref{Alg_LMBM_in_LMBM} we have $\tilde{\bxi}_m=\bxi^{mod}_m = \bxi_m$ and $\tilde{\beta}_m=0$, while $\beta_m$ was set to zero in Step~6 of the previous iteration (with $\beta_1=0$ at initialization). Thus, the base case holds.
  
  Now, assume $k>m$ and let $i \in \{m,\ldots,k-1\}$. Suppose that (\ref{eq_convex_combination_in_inexactLMBM}) is satisfied when $k$ is replaced by $i$. We define the next set of coefficients as
  \begin{align*}
    &\lambda^{i+1,m}= \lambda^i_1+\lambda^i_3\lambda^{i,m},\\
    &\lambda^{i+1,j}= \lambda^i_3\lambda^{i,j} \quad \text{for } j = m+1,\ldots,i, \qquad \text{and} \\
    &\lambda^{i+1,i+1}= \lambda^i_2,
  \end{align*}
  where the values $\lambda_l^i \geq 0$ for $l \in \{1,2,3\}$ are obtained at Step~8 of Algorithm~\ref{Alg_LMBM_in_LMBM}. Now, we have $\lambda^{i+1,j} \geq 0$ for all $j = m,\ldots,i+1$, and
  \begin{align}\label{equation_lambda}
      \sum_{j=m}^{i+1}\lambda^{i+1,j}= \lambda^i_1+ \lambda^i_3\left(
        \lambda^{i,m} + \sum_{j=m+1}^{i}\lambda^{i,j}\right) + \lambda^i_2=1,
  \end{align}
  because $\sum_{j=m}^{i}\lambda^{i,j}=1$ by the induction assumption and $\sum_{l=1}^3
  \lambda^i_l=1$ (Step~8 of Algorithm~\ref{Alg_LMBM_in_LMBM}). Using (\ref{eq_tilde_xi_in_LMBMinexact}), (\ref{eq_tilde_beta_in_LMBMinexact}), and (\ref{equation_lambda}), we obtain
  \begin{align*}
    (\tilde{\bxi}_{i+1},\tilde{\beta}_{i+1}) &=
    \lambda^i_1(\bxi_{m},0)+\lambda^i_2(\bxi^{mod}_{i+1},\beta_{i+1}) +
    \sum_{j=m}^{i}\lambda^i_3\lambda^{i,j}(\bxi^{mod}_j,\beta_j) \\
    &= \sum_{j=m}^{i+1} \lambda^{i+1,j} (\bxi^{mod}_j,\beta_j),
  \end{align*}
  since we have $\beta_{m}=0$ and $\bxi_m = \bxi_m^{mod}$. Therefore, condition~(\ref{eq_convex_combination_in_inexactLMBM}) holds for $i+1$.

   Finally, we define
   \begin{align*}
    \tilde{\sigma}_k=\sum_{j=m}^k \lambda^{k,j} \norm{\y_j-\x_j}.
   \end{align*}
   From (\ref{eq_convex_combination_in_inexactLMBM}), Lemma~\ref{lemmabeta}, and the convexity of the function $g \rightarrow \frac{\gamma}{2} g^2$ on $\R_+$ for $\gamma > 0$, it follows that
\begin{align*}
    &\frac{\gamma}{2} \tilde{\sigma}_k^2 = \frac{\gamma}{2} \left( \sum_{j=m}^k \lambda^{k,j}
      \norm{\y_j-\x_j} \right)^2 \\
    &\phantom{\frac{\gamma}{2} \tilde{\sigma}_k^2} \leq \sum_{j=m}^k \lambda^{k,j}
    \frac{\gamma}{2}
    \norm{\y_j-\x_j}^2 \\
    &\phantom{\frac{\gamma}{2} \tilde{\sigma}_k^2} \leq \sum_{j=m}^k \lambda^{k,j} \beta_j \\
    &\phantom{\frac{\gamma}{2} \tilde{\sigma}_k^2} = \tilde{\beta}_k.
  \end{align*}
\end{proof}

\begin{lemma} 
\label{lemma_q_in_subdiff_in_LMBM}
Let $\bar{\x} \in \R^n$ be given and suppose that there exist vectors $\bar{\bg}$, $\bar{\bxi}_i$, $\bar{\y}_i$, and numbers $r_i \geq 0$, $\bar{\lambda}_i \geq 0$ for $i=1,\ldots,l$, $l \geq 1$, such that
\begin{align}
\label{eq_g_in_subdiff_in_inexactLMBM}
  &(\bar{\bg},0)=\sum_{i=1}^l
  \bar{\lambda}_i(\bar{\bxi}_i,\norm{\bar{\y}_i-\bar{\x}}), \nonumber\\
  &\bar{\bxi}_i \in \partial f(\bar{\y}_i) + B_{r_i},\quad i=1,\ldots,l,\qquad
  \text{and} \\
  &\sum_{i=1}^l \bar{\lambda}_i = 1.\nonumber
\end{align}
Then $\bar{\bg} \in \partial f(\bar{\x}) + B_{\bar r}$, where $\bar{r} \geq r_i$ for $i=1,\ldots,l$.
\end{lemma}
\begin{proof}
Let $\mathcal{I}=\{\,i \mid 1 \leq i \leq l,\,\, \bar{\lambda}_i>0\,\}$. By (\ref{eq_g_in_subdiff_in_inexactLMBM}) we have $$
\bar{\y}_i = \bar{\x} \quad \text{and} \quad
\bar{\bxi}_i \in \partial f(\bar{\x})+B_{r_i}
$$
for all $i \in \mathcal{I}$. Write $\bar{\bxi}_i=\bg_i+\be_i$, where $\bg_i \in \partial f(\bar{\x})$ and $\be_i \in B_{r_i}$, meaning that $\norm{\be_i}\leq r_i$. Therefore,
  \begin{align*}
    &\bar{\bg}=\sum_{i \in \mathcal{I}} \bar{\lambda}_i \bar{\bxi}_i = \sum_{i \in \mathcal{I}} \bar{\lambda}_i \bg_i + \sum_{i \in \mathcal{I}} \bar{\lambda}_i \be_i,\\
    &\bar{\lambda}_i>0,\qquad\text{for } i \in \mathcal{I},\qquad \text{and} \\
    &\sum_{i \in \mathcal{I}} \bar{\lambda}_i=1.
  \end{align*}
  By the convexity of $\partial f(\bar{\x})$ (see \cite{BagKarMak:2014}), it follows that $\sum_{i \in \mathcal{I}} \bar{\lambda}_i \bg_i \in \partial f(\bar{\x})$. Moreover, by the triangle inequality and the fact that $\norm{\be_i}\leq r_i \leq \bar{r}$ and $\sum_{i \in \mathcal{I}} \bar{\lambda}_i=1$,
  $$\norm{\sum_{i \in \mathcal{I}} \bar{\lambda}_i \be_i} \leq \sum_{i \in \mathcal{I}} \bar{\lambda}_i \norm{\be_i} \leq \sum_{i \in \mathcal{I}} \bar{\lambda}_i \bar{r} = \bar{r},$$
  which implies that $\sum_{i \in \mathcal{I}} \bar{\lambda}_i \be_i \in B_{\bar{r}}$. Thus, $\bar{\bg}\in \partial f(\bar{\x})+B_{\bar{r}}$.
\end{proof}

\begin{theorem} 
\label{theo_terminate}
If Algorithm~\ref{Alg_LMBM_in_LMBM} terminates at the $k$th iteration, then
the point $\x_k$ is approximately stationary for $f$.
\end{theorem}
\begin{proof}
  If Algorithm~\ref{Alg_LMBM_in_LMBM} terminates at Step~4, then the selection $\varepsilon=0$ implies that $w_k=0$. By Lemma~\ref{lemma_k_th_iteration_in_LMBM}, this gives $\tilde{\bxi}_k={\pmb 0}$ and $\tilde{\beta}_k=0$. Furthermore, let $m \leq k$ be the index of the iteration after the latest serious step. Using Lemma~\ref{lemma_lambda_convex_combination_in_LMBM} and denoting $\mathcal{J}=\{j \; | \; m \leq j \leq k, \; \lambda^{k,j}>0\}$, we know that $\tilde{\sigma}_k=0$ and $$(\tilde{\bxi}_k,\tilde{\sigma}_k) = \sum_{j\in \mathcal{J}} \lambda^{k,j}
  (\bxi^{mod}_j,\norm{\y_j-\x_j}),$$
  where $\sum_{j\in \mathcal{J}} \lambda^{k,j} = 1$. These together yield that $\norm{\y_j-\x_j}=0$ for $j\in \mathcal{J}$, meaning that $\y_j=\x_j$ and $\bxi^{mod}_j=\bxi_j \in \partial f(\y_j)+B_{r_j}$. Therefore, 
  $$(\tilde{\bxi}_k,0) = \sum_{j=m}^k \lambda^{k,j} (\bxi_j,\norm{\y_j-\x_k}),$$
  since $\x_j=\x_k$ for $j=m,\ldots,k$ and 
  $\lambda^{k,j}=0$ for $j \in \{m, \ldots, k\}\setminus \mathcal{J}$.
  
  Applying Lemma~\ref{lemma_q_in_subdiff_in_LMBM} with the choices    
  \begin{alignat*}{20}
    &\bar{\x}=\x_k, \qquad\qquad  && l=k-m+1, \quad\qquad &&& \bar{\bg}=\tilde{\bxi}_k, \quad\qquad &\quad r_i = r_{i+m-1},\\
    &\bar{\bxi}_i=\bxi_{i+m-1}, &&\bar{\y}_i=\y_{i+m-1}, &&&\bar{\lambda}_i=\lambda^{k,i+m-1}
  \end{alignat*}
  for $i=1,\ldots ,l$, it follows that ${\pmb 0}=\tilde{\bxi}_k \in \partial f(\x_k) + B_{\bar{r}}$. Consequently, $\x_k$ is approximately stationary for $f$.
\end{proof}

From this point onward, we assume that Algorithm~\ref{Alg_LMBM_in_LMBM} continues without termination. In other words, $w_k>0$ for all $k$.

%\begin{lemma} 
%\label{lemma_w_k->0->q_k->0_in_LMBM}
%Let the stopping parameters $w_k$ and $q_k$ of Algorithm~\ref{Alg_LMBM_in_LMBM} be defined by (\ref{eq_w_k_in_LMBM}) and (\ref{eq_q_k_in_LMBM}), respectively. Then
%\begin{align*}
%  \{q_k\} \rightarrow 0 \qquad \text{if and only if} \qquad \{w_k\} \rightarrow 0.
%\end{align*}
%\end{lemma}
%
%\begin{proof}
%  See the proof of Lemma 5 in \cite{HaaMieMak:2007}.

%  {\color{blue}
%  The condition $\{q_k\} \rightarrow 0$ implies $\{\tilde{\bxi}_k\} \rightarrow {\pmb 0}$ and $\{\tilde{\beta}_k\} \rightarrow 0$ by Lemma~\ref{lemma_k_th_iteration_in_LMBM} and, thus, we have $\{w_k\} \rightarrow 0$.
  
%  On the other hand, by Lemma~\ref{lemma_k_th_iteration_in_LMBM} we have $w_k \geq 2\tilde{\beta}_k \geq 0$ and $w_k \geq \varrho \norm{\tilde{\bxi}_k}^2$ for some correction parameter $\varrho \in (0,1/2)$. Therefore, $\{w_k\} \rightarrow 0$ implies $\{\tilde{\beta}_k\} \rightarrow 0$ and $\{\tilde{\bxi}_k\} \rightarrow {\pmb 0}$ and, thus, also $\{q_k\} \rightarrow 0$.
%}
%\end{proof}

%Based on Lemma~\ref{lemma_w_k->0->q_k->0_in_LMBM}, it is sufficient from this point onward to focus only on the stopping parameter $w_k$.

\begin{lemma} 
\label{lemma_boundedness_and_stationary_in_LMBM}
Sequences $\{\x_k\}$, $\{\y_k\}$, $\{\bxi_k\}$, $\{\bxi^{mod}_k\}$, $\{\eta_k\}$ and $\{r_k\}$ are bounded. If there exist a point $\bar{\x} \in \R^n$ and an infinite set $\K
\subset \{1,2,\ldots\}$ such that $\{\x_k\}_{k \in \K} \rightarrow \bar{\x}$ and 
$\{w_k\}_{k \in \K} \rightarrow 0$, then ${\pmb 0} \in \partial f(\bar{\x}) + B_{\bar{r}}$, where $\bar{r}$ is the error bound defined in Assumption \ref{assumption_errors}.
\end{lemma}
\begin{proof}
  Since the sequence $\{\hat{f}_k\}$ is monotone due to the sufficient descent criterion \eqref{eq_serious_descent_criterion_in_LMBM}, the sequence $\{\x_k\}$ belongs to the level set $\mathcal{F}(\x_1)$. Together with Assumption \ref{assumption_levelset} this implies that $\{\x_k\}$ is bounded. The scaling of the direction vector in Step~5 of Algorithm~\ref{Alg_LMBM_in_LMBM}, whenever necessary, guarantees that always $\norm{\bd_k} \leq C$. The auxiliary point is given by $\y_{k+1} = \x_k + t_k \bd_k$ and a stepsize $t_k \in [t_{min},1]$, where $t_{min} \in (0,1]$. Therefore, $$
  \norm{\y_{k+1} - \x_{k}} = t_k \norm{\bd_k} \leq \norm{\bd_k} \leq C.
  $$
  As a result, the sequence $\{\y_k\}$ remains bounded as well.

  The sequences $\{\eta_k\}$ and $\{r_k\}$ are bounded by Assumptions \ref{assumption_eta} and \ref{assumption_errors}. Moreover, by Assumption \ref{assumption_LLC}, $\partial f$ is locally bounded and upper semicontinuous (see, e.g.,~\cite{MakNei:1992}). This, together with the fact $\bxi_k \in \partial f(\y_k) + B_{r_k}$ and the boundedness of $\{\y_k\}$ and $\{r_k\}$, implies that the sequence $\{\bxi_k\}$ is bounded. Since $\bxi^{mod}_k=\bxi_k + \eta_k (\y_k - \x_k)$, and we already know that $\{\bxi_k\}$, $\{\eta_k\}$, $\{\y_k\}$ and $\{\x_k\}$ are bounded, it follows that also $\{\bxi^{mod}_k\}$ is bounded.
  
  Let
  \begin{align*}
    \mathcal{I}=\{1,\ldots,n+2\}
  \end{align*}
  and $m$ denote the index of the iteration after the latest serious step.
  Using the relation $\bxi^{mod}_k=\bxi_k + \eta_k (\y_k - \x_k)$, where $\bxi_k \in \partial f(\y_k)+B_{r_k}$ for all $k \geq 1$, together with Lemma~\ref{lemma_lambda_convex_combination_in_LMBM}, Step~1 of Algorithm~\ref{Alg_LMBM_in_LMBM}, and Carath\'eodory's theorem (see, e.g.,~\cite{HirLem:1993a}), we deduce that there exist vectors $\y_{k,i}$ and $\bxi^{mod}_{k,i}$, and scalars $\lambda^{k,i} \geq 0$ and
  $\tilde{\sigma}_k$ for $i \in \mathcal{I}$ and $k \geq m$, satisfying
  \begin{align}
    \label{eq_boundedness_and_stationary_in_LMBMinexact}
    &(\tilde{\bxi}_k,\tilde{\sigma}_k) =\sum_{i \in \mathcal{I}} \lambda^{k,i}
    (\bxi^{mod}_{k,i},\norm{\y_{k,i}-\x_k}),\nonumber \\
    &\bxi^{mod}_{k,i}=\bxi_{k,i} + \eta_{k,i} (\y_{k,i} - \x_{k}),\\
    & \bxi_{k,i} \in \partial f(\y_{k,i})+B_{r_{k,i}}, \quad\text{and} \nonumber \\
    &\sum_{i \in \mathcal{I}} \lambda^{k,i} = 1,\nonumber \intertext{with}
    &(\y_{k,i},\bxi^{mod}_{k,i}) \in \{(\y_j,\bxi_j)\mid j=m,\ldots,k\}.\nonumber
  \end{align}
    
  Because $\{\y_k\}$ is bounded, there exist points $\y_i^*$ for $i \in \mathcal{I}$ and an infinite set $\K_0 \subset \K$ such that $\{\y_{k,i}\}_{k \in \K_0} \rightarrow \y_i^*$ for $i \in \mathcal{I}$. The boundedness of the sequences $\{\bxi_k\}$, $\{\lambda^{k,i}\}$, $\{\eta_{k}\}$ and $\{r_{k}\}$ ensures the existence of vectors
  $\bxi_i^* \in \partial f(\y_i^*)+B_{r_i^*}$, scalars $\lambda_i^*$, $\eta_i^*$ and $r_i^*$ for $i \in \mathcal{I}$, together with an infinite set $\K_1 \subset \K_0$ such that $\{\bxi_{k,i}\}_{k \in \K_1}
  \rightarrow \bxi_i^*$, $\{\lambda^{k,i}\}_{k \in \K_1} \rightarrow
  \lambda_i^*$, $\{\eta_{k,i}\}_{k \in \K_1}
  \rightarrow \eta_i^*$ and $\{r_{k,i}\}_{k \in \K_1}
  \rightarrow r_i^*$ for $i \in \mathcal{I}$. In addition, since $\{\x_k\}_{k \in \K} \rightarrow \bar{\x}$, it also holds that $\{\x_k\}_{k \in \K_1} \rightarrow \bar{\x}$. Now, we know that $\bxi^{mod}_{k,i}=\bxi_{k,i} + \eta_{k,i} (\y_{k,i} - \x_{k})$ and all its components converge. Thus, $\bxi^{mod}_{k,i}$ converges to $\bxi^{mod,*}_{i}=\bxi_{i}^* + \eta_{i}^* (\y_{i}^* - \bar{\x})$.
  
  Since $\{w_k\}_{k \in \K} \rightarrow 0$,
  Lemma~\ref{lemma_k_th_iteration_in_LMBM} together with
  Lemma~\ref{lemma_lambda_convex_combination_in_LMBM} implies
  $$
  \{\tilde{\bxi}_k\}_{k \in \K} \rightarrow {\pmb 0}, \qquad
  \{\tilde{\beta}_k\}_{k \in \K} \rightarrow 0, \qquad \text{and}\qquad
  \{\tilde{\sigma}_k\}_{k \in \K} \rightarrow 0.
  $$
  Thus, these sequences converge also in the subset $\mathcal{K}_1$ and from (\ref{eq_boundedness_and_stationary_in_LMBMinexact}) we obtain
 \begin{align*}
    & (\mathbf{0},0) =\sum_{i \in \mathcal{I}} \lambda^*_i
    (\bxi^{mod,*}_{i},\norm{\y_{i}^*-\bar{\x}}), \\
    & \bxi^{mod,*}_{i}=\bxi_{i}^* + \eta_{i}^* (\y_{i}^* - \bar{\x}), \\
    & \bxi_{i}^* \in \partial f(\y_{i}^*)+B_{r_{i}^*}, \\
    & \lambda_i^*\geq0 \quad\text{for } i \in \mathcal{I}, \quad \text{and} \\
    & \sum_{i \in \mathcal{I}} \lambda_i^* = 1.
 \end{align*}
  If $\lambda^{i*}>0$, then $\norm{\y_{i}^*-\bar{\x}}=0$. Therefore, $\y_i^{*}=\bar{\x}$, $\bxi^{mod,*}_i = \bxi_i^*$, and $\bxi^{mod,*}_i \in \partial f(\bar{\x}) + B_{r_i^*}$, where $r_i^*\leq \bar{r}$ by Assumption \ref{assumption_errors}.
  
  Finally, letting $k \in \K_1$ approach infinity in (\ref{eq_boundedness_and_stationary_in_LMBMinexact}) and applying
  Lemma~\ref{lemma_q_in_subdiff_in_LMBM} with
  \begin{alignat*}{5}
    & l=n+2, \qquad && \bar{\bg}={\pmb 0}, \qquad &&&\bar{\bxi}_i=\bxi_i^*,\\
    & \bar{\y}_i=\y_i^*, &&
    \bar{\lambda}_i=\lambda_i^*, \qquad &&& r_i=r_i^*,
  \end{alignat*}
  we conclude that ${\pmb 0} \in \partial f(\bar{\x}) + B_{\bar{r}}$, which completes the proof.
\end{proof}

\begin{lemma} 
\label{lemma_D_(k+1)_leq_D_k_and_TrD_(k+1)_leq_n_in_LMBM}
Suppose that the number of serious steps is finite and the last serious step occurred at the iteration $m-1$.  Then there exists a number $k^* \geq m$,
such that
\begin{align}
  \label{eq_D_(k+1)_leq_D_k_in_LMBM}
  &\tilde{\bxi}_{k+1}^{\top} D_{k+1} \tilde{\bxi}_{k+1} \leq \tilde{\bxi}_{k+1}^{\top}
  D_{k} \tilde{\bxi}_{k+1} \qquad \text{and}\\
  \label{eq_TrD_(k+1)_leq_n_in_LMBM}
  &\tr(D_{k}) < \frac{3}{2}n 
\end{align}
for all $k \geq k^*$, where $\tr(D_k)$ denotes the trace of matrix $D_k$.
\end{lemma}
\begin{proof}  
  See the proof of Lemma 7 in \cite{HaaMieMak:2007}.
\end{proof}

\begin{lemma}
\label{lemma_minQ_leq_w_in_LMBM}
Suppose that there exist vectors $\bp$ and $\bg$ together with numbers $w \geq
0$, $\alpha \geq 0$, $\beta \geq 0$, $M \geq 0$, and $c \in (0,1/2)$ such that
\begin{align*}
  w=\norm{\bp}^2+2\alpha, \quad \beta+\bp^{\top} \bg \leq cw, \quad \text{and}\quad
  \max \,\{\norm{\bp},\norm{\bg},\sqrt{\alpha}\} \leq M.
\end{align*}
Let $Q:[0,1]\rightarrow \R$ be such that
\begin{align*}
  &Q(\lambda)=\norm{\lambda \bg + (1-\lambda) \bp}^2 + 2(\lambda \beta +
  (1-\lambda) \alpha) \qquad \text{and}\\
  &b=(1-2c)/4M.  \intertext{Then} &\min \,\{Q(\lambda) \mid \lambda \in
  [0,1]\}\leq w -w^2 b^2.
\end{align*}
\end{lemma}
\begin{proof}
    See the proof of Lemma 3.5 in \cite{LukVlc:1999}.
\end{proof}

\begin{lemma} 
\label{lemma_finite_serious_steps_in_LMBM}
Suppose that the number of serious steps is finite and the last serious step occurred at the iteration $m-1$. Then, the point $\x_{m}$ is approximately stationary for~$f$.
\end{lemma}
\begin{proof}
  Using (\ref{eq_quad_prob_in_LMBMinexact})--(\ref{eq_tilde_beta_in_LMBMinexact}),
  Lemma~\ref{lemma_k_th_iteration_in_LMBM}, and Lemma~\ref{lemma_D_(k+1)_leq_D_k_and_TrD_(k+1)_leq_n_in_LMBM} we obtain
  \begin{align}
    \label{eq_w_(k+1)_leq_w_k_in_LMBMinexact}
    w_{k+1} &= \tilde{\bxi}_{k+1}^{\top} D_{k+1} \tilde{\bxi}_{k+1} +
    2\tilde{\beta}_{k+1} \nonumber \\
    &\leq \tilde{\bxi}_{k+1}^{\top} D_k \tilde{\bxi}_{k+1} + 2\tilde{\beta}_{k+1} \nonumber\\
    &=\varphi (\lambda^k_1, \lambda^k_2, \lambda^k_3)\\
    &\leq\varphi (0,0,1)\nonumber\\
    &=\tilde{\bxi}_k^{\top} D_k \tilde{\bxi}_k + 2\tilde{\beta}_k \nonumber \\
    &=w_k\nonumber
  \end{align}
  for all $k \geq k^*$, where $k^*$ is defined in Lemma~\ref{lemma_D_(k+1)_leq_D_k_and_TrD_(k+1)_leq_n_in_LMBM}.
  
  For convenience, write $D_k=W_k^{\top} W_k$. With this notation, the function $\varphi$ defined in (\ref{eq_quad_prob_in_LMBMinexact}) can be expressed as
  \begin{align*}
    \varphi (\lambda^k_1, \lambda^k_2, \lambda^k_3)= \norm{\lambda^k_1
      W_k\bxi_m + \lambda^k_2 W_k\bxi_{k+1}^{mod} + \lambda^k_3 W_k\tilde{\bxi}_k}^2
    + 2(\lambda^k_2 \beta_{k+1} + \lambda^k_3 \tilde{\beta}_k).
  \end{align*}
   Relation (\ref{eq_w_(k+1)_leq_w_k_in_LMBMinexact}) implies that the sequences $\{w_k\}$, $\{W_k\tilde{\bxi}_k\}$, and $\{\tilde{\beta}_k\}$ remain bounded. Furthermore, Lemma~\ref{lemma_D_(k+1)_leq_D_k_and_TrD_(k+1)_leq_n_in_LMBM}
  guarantees boundedness of $\{D_k\}$ and $\{W_k\}$, while Lemma~\ref{lemma_boundedness_and_stationary_in_LMBM} ensures that the sequences $\{\y_k\}$, $\{\bxi_k\}$, $\{\bxi_k^{mod}\}$, $\{\eta_k\}$ and $\{r_k\}$ are also bounded. Consequently, the sequence $\{W_k \bxi_{k+1}^{mod}\}$ is bounded as well.
  
  Define
  \begin{align*}
     M=\sup \,\{\norm{W_k \bxi_{k+1}^{mod}},\norm{W_k
      \tilde{\bxi}_k},\sqrt{\tilde{\beta}_k} \mid k \geq k^*\},
  \end{align*}
  and set
  \begin{align*}
     b=(1-2\varepsilon_L)/4M.
  \end{align*}
  Assume first that $w_k>\delta>0$ holds for all $k \geq k^*$. Because
  \begin{align*}
    &\min \,\{\varphi (\lambda_1, \lambda_2, \lambda_3) \mid \lambda_i \geq
    0,\,
    i=1,2,3,\, \sum_{i=1}^3 \lambda_i=1\} \\
    &\leq \min \,\{\varphi(0,\lambda,(1-\lambda)) \mid \lambda \in [0,1]\},
  \end{align*}
  relation (\ref{eq_w_(k+1)_leq_w_k_in_LMBMinexact}) yields
  \begin{align*}
    w_{k+1} \leq \min \,\{\norm{\lambda W_k \bxi_{k+1}^{mod} + (1-\lambda) W_k
      \tilde{\bxi}_k)}^2 + 2(\lambda \beta_{k+1} + (1-\lambda)
    \tilde{\beta}_k) \mid \lambda \in [0,1]\}.
  \end{align*}
  From Lemma~\ref{lemma_k_th_iteration_in_LMBM}, it follows that $w_k=\tilde{\bxi}_k^{\top} D_{k} \tilde{\bxi}_k + 2\tilde{\beta}_k$. Moreover, $\bd_k=-D_k \tilde{\bxi}_k$ (see Step~2 in Algorithm~\ref{Alg_LMBM_in_LMBM}),
  and condition (\ref{cond_null_step_in_LMBMinexact}) implies
  \begin{align*}
    -\beta_{k+1} + t_k \bd_k^{\top} \bxi_{k+1}^{mod} \geq -\varepsilon_L t_k w_k.
  \end{align*}
  When we take into account that $t_k \in [t_{min}, 1]$, where $t_{min} \in (0,1]$, we have
  \begin{align*}
    t_k \beta_{k+1} - t_k \bd_k^{\top} \bxi_{k+1}^{mod} \leq \beta_{k+1} -
    t_k \bd_k^{\top} \bxi_{k+1}^{mod} \leq \varepsilon_L t_k w_k.
  \end{align*}
  
  Hence, Lemma~\ref{lemma_minQ_leq_w_in_LMBM} can be applied with
  \begin{alignat*}{5}
    &\bp=W_k \tilde{\bxi}_k, \qquad && \bg=W_k \bxi_{k+1}^{mod}, \qquad && w=w_k,\\
    &\alpha=\tilde{\beta}_k, && \beta=\beta_{k+1}, && c=\varepsilon_L,
  \end{alignat*}
  which leads to
  \begin{align*}
    w_{k+1}\leq w_k - (w_k b)^2 < w_k - (\delta b)^2
  \end{align*}
  for all $k \geq k^*$. For sufficiently large $k$, this contradicts the assumption $w_k > \delta$. Therefore, using the monotonicity of $w_k$ for $k \geq k^*$, we conclude that $w_k \rightarrow 0$ and $\x_k
  \rightarrow \x_m$.
  
  Finally, Lemma~\ref{lemma_boundedness_and_stationary_in_LMBM}, implies that ${\pmb 0} \in
  \partial f(\x_m) + B_{\bar{r}}$, where $\bar{r}$ is the error bound defined in Assumption \ref{assumption_errors}. Thus $\x_m$ is an approximate stationary point of $f$.
\end{proof}

\begin{theorem} 
\label{theo_cluster_points_are_stationary_in_LMBM}
Every accumulation point of the sequence $\{\x_k\}$ is approximately stationary for $f$.
\end{theorem}
\begin{proof}
  Let $\bar{\x}$ be an accumulation point of the sequence $\{\x_k\}$. Then there exists an infinite set $\K \subset
  \{1,2,\ldots\}$ such that $\{\x_k\}_{k \in \K} \rightarrow
  \bar{\x}$.
  According to Lemma~\ref{lemma_finite_serious_steps_in_LMBM}, if the number of serious steps were finite, the final serious iterate would already be an approximate stationary point. Hence it is sufficient to consider only the case where infinitely many serious steps occur.
  
  Let $\K'\subset \{1,2,\ldots\}$ denote the set of indices corresponding to serious steps and define
  \begin{align*}
    \K''=\{k\in\K' \mid \exists i \in \K,\,\, i \leq k
    \text{ such that } \x_i=\x_k\}.
  \end{align*}
  The set $\K''$ is clearly infinite, and the subsequence $\{\x_k\}_{k \in \K''}$ still converges to $\bar{\x}$.
  By Assumption \ref{assumption_LLC}, $f$ is continuous and it follows that $\{\hat{f}_k\}_{k \in \K''} \rightarrow f(\bar{\x})$. Moreover, because the sequence $\{\hat{f}_k\}$ is monotonically decreasing due to the sufficient descent criterion~(\ref{eq_serious_descent_criterion_in_LMBM}), we obtain $\hat{f}_k \downarrow f(\bar{\x})$.
  This information together with the condition~(\ref{eq_serious_descent_criterion_in_LMBM}), gives us
  \begin{align}
\label{eq_serious_descent_criterion2_in_LMBMinexact}
  0 \leq \varepsilon_L t_k w_k \leq \hat{f}_k - \hat{f}_{k+1} \rightarrow 0 \qquad
  \text{for }k \geq 1.
  \end{align}
  Consequently, relation (\ref{eq_serious_descent_criterion2_in_LMBMinexact}) implies that $\{w_k\}_{k \in \K''} \rightarrow 0$, when $\{\x_k\}_{k \in \K''} \rightarrow \bar{\x}$, since $t_k \geq t_{min} > 0$ and $\varepsilon_L > 0$. 
  Finally, applying Lemma~\ref{lemma_boundedness_and_stationary_in_LMBM} yields that ${\pmb 0} \in
  \partial f(\bar{\x})+B_{\bar{r}}$, where $\bar{r}$ is the error bound defined in Assumption \ref{assumption_errors}. Hence $\bar{\x}$ is an approximate stationary point of $f$.
\end{proof}

In summary, the sequence $\{\x_k\}$ generated by Algorithm~\ref{Alg_LMBM_in_LMBM} approaches approximate stationarity and we have proved the global convergence of the proposed method. In addition, it is worth noting that global convergence also holds when exact function and subgradient values are used. In this case, we simply set the error bounds $\bar{q}=0$ and $\bar{r} = 0$, and the method converges to a stationary point.

\section{Numerical Experiments} \label{sec_Numerical_Experiments}
We first compare the proposed inexact limited memory bundle method ({\tt InexactLMBM}) with the original limited memory bundle method ({\tt LMBM}) \cite{HaaMieMak:2004,HaaMieMak:2007} and with the proximal bundle method ({\tt MPBNGC}) \cite{Mak:2003,MakNei:1992} on standard academic nonconvex test problems from \cite{HaaMieMak:2004} and on Ferrier polynomials from \cite{HarSagSol:2016} (see Appendix~\ref{appendix_testproblems}). Since both {\tt LMBM} and {\tt MPBNGC} assume exact function and subgradient information, we benchmark against them only on problems with exact evaluations. We then assess the effect of inexactness by running {\tt InexactLMBM} on the same problems with varying noise types and levels (bounds). The source code (Fortran 95) of the proposed method is available at \url{https://github.com/napsu/InexactLMBM}.

% I have 4 core / 8 thread CPU
All computational experiments are carried out on iMac, 4.0 GHz Quad-Core Intel(R) Core(TM) i7 machine with 16 GB of RAM. We use gfortran to compile the Fortran codes. 

\paragraph*{Benchmarking solvers.}
Because {\tt InexactLMBM} is a modification of {\tt LMBM}, {\tt LMBM} is the natural baseline for quantifying the effect of the proposed changes. We note that the original {\tt LMBM} is designed for nonsmooth, large-scale optimization \cite{HaaMieMak:2004,HaaMieMak:2007} and remains one of the few general-purpose solvers for high-dimensional nonsmooth problems. 
In our
experiments, we used the adaptive version of the code with the initial number of stored correction
pairs (used to form the variable metric update) equal to 7 and the maximum number of stored
correction pairs equal to 15.
The {\tt LMBM} source code (Fortran 95) is available at \url{https://napsu.karmitsa.fi/ldgbm/}.
% Jos tähän adaptiiviseen versioon haluaa viitteen, niin pitää viitata minun väitöskirjaani. Muissa missä sitä on käytetty on vain epämääräissesti sanottu, että käytetään adaptiivista versiota (yleensä ilman viitettä).

We also benchmark against {\tt MPBNGC} because the proximal bundle method is the most widely used bundle scheme in nonsmooth optimization, and {\tt MPBNGC} is a mature well-documented implementation \cite{Mak:2003,MakNei:1992}. %\cite{Mak:1993}. 
%The benchmark solver is the proximal bundle algorithm {\tt MPBNGC} from the {\tt NSOLIB} (NonSmooth Optimization LIBrary; see~\cite{Mak:1993}). 
{\tt MPBNGC}  implements the subgradient aggregation strategy of~\cite{Kiw:1985} and uses the quadratic program solver  {\tt PLQDF1}, developed by Luk{\v s}an -- a dual active-set method~\cite{Luk:1984} -- to solve the quadratic direction-finding subproblem; see \cite{Mak:1993,MakNei:1992} for details. We use {\tt MPBNGC} (version 4.0), whose source code (Fortran 77) is available at \url{https://napsu.karmitsa.fi/proxbundle/}; this version includes an improved termination criterion that detects stagnation when objective function values no longer change.

A more extensive numerical analysis of the performance of these two benchmarking solvers and some other existing bundle methods can be found in \cite{HaaMieMak:2004,KarBagMak:2012}.

\paragraph*{Parameters.}
Following \cite{KarBagMak:2012}, we set the bundle size to $\min\{n+3,100\}$ for all solvers. However, it is worth noting that {\tt LMBM} and {\tt InexactLMBM} use only {\em two} bundle elements (together with the values at the current iteration point) to compute aggregate values;  the larger bundle is used solely for (initial) step size selection. We used $\gamma=0.5$ and the stopping tolerance $\varepsilon=10^{-5}$ throughout.   In addition, all solvers include stagnation-based termination: they stop if either the function value or the subgradient norm remains unchanged (within a small numerical tolerance) for a prescribed number of consecutive iterations; we used the default thresholds provided by the implementations. The maximum number of iterations was set to $10{,}000$ for all solvers.\footnote{For {\tt InexactLMBM} this cap corresponds to the number of function evaluations, whereas for {\tt LMBM} and {\tt MPBNGC} the number of function evaluations may be larger.} Otherwise, the default parameters (see implementations of the methods) are used with all methods.

For {\tt InexactLMBM}, the defaults include $t_{min}=10^{-12}$, $\varepsilon_L = 0.01$, $C = 10^{20}$, $\varrho = 10^{-12}$, and a nonmonotone step size selection with \emph{three} previous function values.
Similarly to {\tt LMBM}, we used an adaptive number of correction vectors $\hat{m}_c$, storing 7 pairs initially and 15 pairs at maximum.

\paragraph*{Noise models.} We evaluate {\tt InexactLMBM} under five different noise settings. At each evaluation, the algorithm receives perturbed quantities:
\begin{align*}
    &f_k = f(\y_k)-q_k,\\
    &\bxi_k \in \partial f(\y_k) + B_{r_k}, 
\end{align*}
where $|q_k| \leq q^f_k \leq \bar{q}$ and $0 \leq r_k \leq q^\xi_k \leq \bar{q}$, the bounds $q^f_k$ and $q^\xi_k$ may depend on $k$, and $\bar{q} \geq 0$. The tested forms of noise are: 
\begin{itemize}
    %\item $N_0$: 
    \item $\mathcal{N}0$: no noise, $\bar{q}=q^f_k = q^\xi_k = 0$ for all $k$; %ok
    
    %\item $N_{c}^{f,\xi}$: 
    \item $\mathcal{N}1$: constant noise on both function values and subgradients, $\bar{q} > 0$ fixed, $q^f_k=\bar{q}$, and $q^\xi_k=\bar{q}$ for all $k$; %ok
    
    %\item $N_{v}^{f,\xi}$: 
    \item $\mathcal{N}2$: vanishing noise on both function values and subgradients, $\bar{q} > 0$ fixed, $q^f_k \leq \bar{q}$, $q^\xi_k \leq \bar{q}$, and sequences $q^f_k \downarrow 0, q^\xi_k \downarrow 0$ when $k \rightarrow \infty$;
    
    %\item $N_{c}^{\xi}$: 
    \item $\mathcal{N}3$: exact values for functions with constant noise on subgradient; $q^f_k = 0$, $\bar{q} > 0$ fixed, and $q^\xi_k=\bar{q}$ for all $k$;
    
    %\item $N_{v}^{\xi}$: 
    \item $\mathcal{N}4$: exact values for functions with vanishing subgradient noise, $q^f_k = 0$, $\bar{q} > 0$ fixed, $q^\xi_k \leq \bar{q}$, and the sequence $q^\xi_k \downarrow 0$ when $k \rightarrow \infty$;
    
\end{itemize}
{Similarly to \cite{HarSagSol:2016} vanishing noises are computed by the formulae
\begin{align*}
&q^f_k= \min \{ \bar{q}, \norm{\x_k-\x^*} / 100 \}, \\
&q_\xi^k= \min \{ \bar{q}, \norm{\x_k-\x^*}^2 / 100 \}, 
\end{align*}
where $\x^*$ is the optimal (or best-known) solution. Note that in all cases but $f_3$ (see Appendix~\ref{appendix_testproblems}), $\x^* = 0.0$. For $f_3$ we used the point $\x^*$ obtained with {\tt InexactLMBM} when no noise was included to computations.
}

\begin{remark} (\emph{Boundedness of the convexification parameter $\eta$ under inexact information}) In the exact case, the convexification parameter sequence $\eta_k$ remains bounded  under reasonable assumptions (see Lemma 3 in \cite{HarSag:2010}). 
In the inexact setting, no general boundedness guarantee is available without additional assumptions on the error behavior \cite{HarSagSol:2016}. Nevertheless, in our numerical experiments -- and likewise in the computational study of Hare et al.\ \cite{HarSagSol:2016} -- we did not encounter numerical difficulties attributable to this issue; such adversarial configurations appear rare and predominantly artificial in practice.
\end{remark}

\paragraph*{Problems with exact information.} \texttt{InexactLMBM} was first compared with the original {\tt LMBM} and the proximal bundle method {\tt MPBNGC} using \emph{five} nonsmooth nonconvex academic minimization problems from \cite{HaaMieMak:2004} and \emph{five} Ferrier polynomials from \cite{HarSagSol:2016} with exact information in both function values and subgradients (see also Appendix~\ref{appendix_testproblems}). 
The number of variables used were 2, 5, 10, 20, 50, 100, 200, 500, 1000, and 2000 which gives us 100 nonsmooth nonconvex test problems.

Results with exact information are summarized in Figure \ref{fig:exact_nonconvex} (see also Figure \ref{fig:exact_comparison_pairwise} in Appendix~\ref{appendix_results}). In Figure \ref{fig:error} relative errors (lower is better) with respect to the global solution (or the best known) are given. Figure \ref{fig:accuracy} plots the accuracy of algorithms (higher is better) given in form
$$
\text{Accuracy} = - \log_{10} \left( \max(f-f_{best},10^{-10})\right).
$$
Figure \ref{fig:Nf} shows the number of function evaluations for the algorithms (lower is better) and Figure \ref{fig:cpu} gives the elapsed computational times (lower is better).

\begin{figure*}[h!] 
\centering
\subfigure[Relative error]{\label{fig:error}
\includegraphics[width=0.49\textwidth]{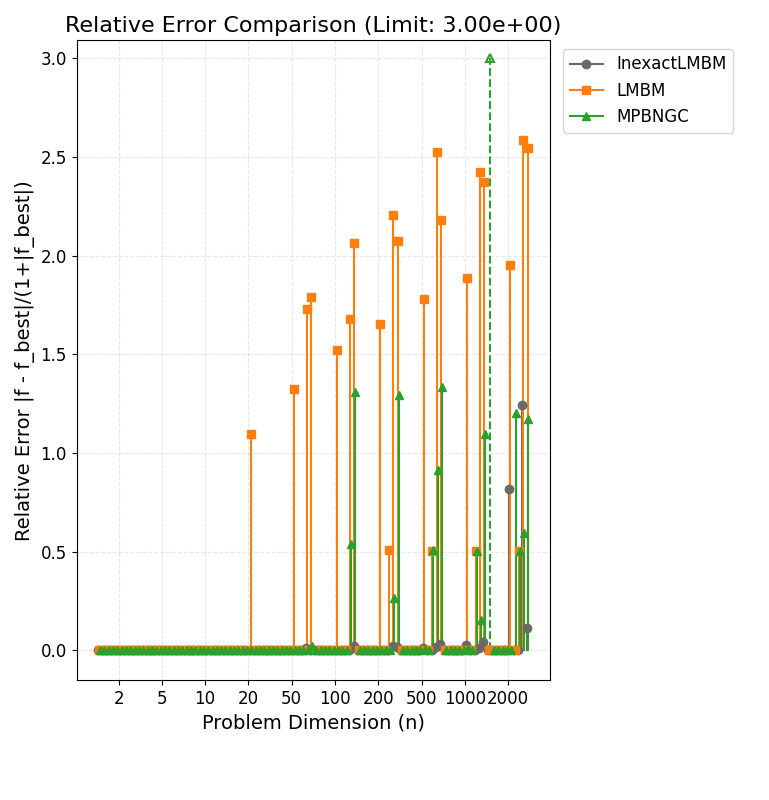}}%\hspace{-1.7cm}
\subfigure[Accuracy]{\label{fig:accuracy}
\includegraphics[width=0.49\textwidth]{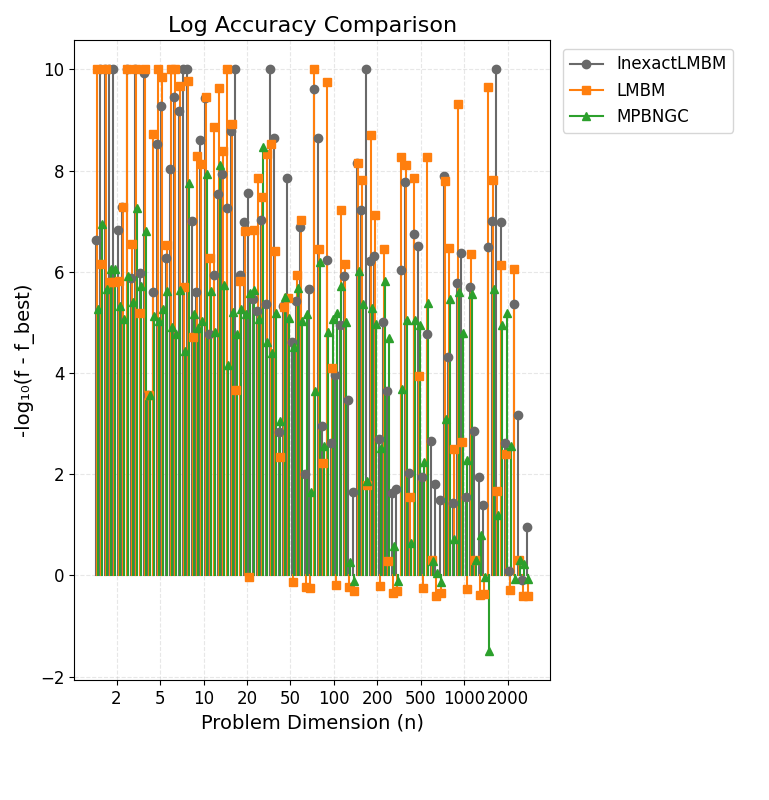}}%\hspace{-1.7cm}

\subfigure[Function evaluations]{\label{fig:Nf}
\includegraphics[width=0.49\textwidth]{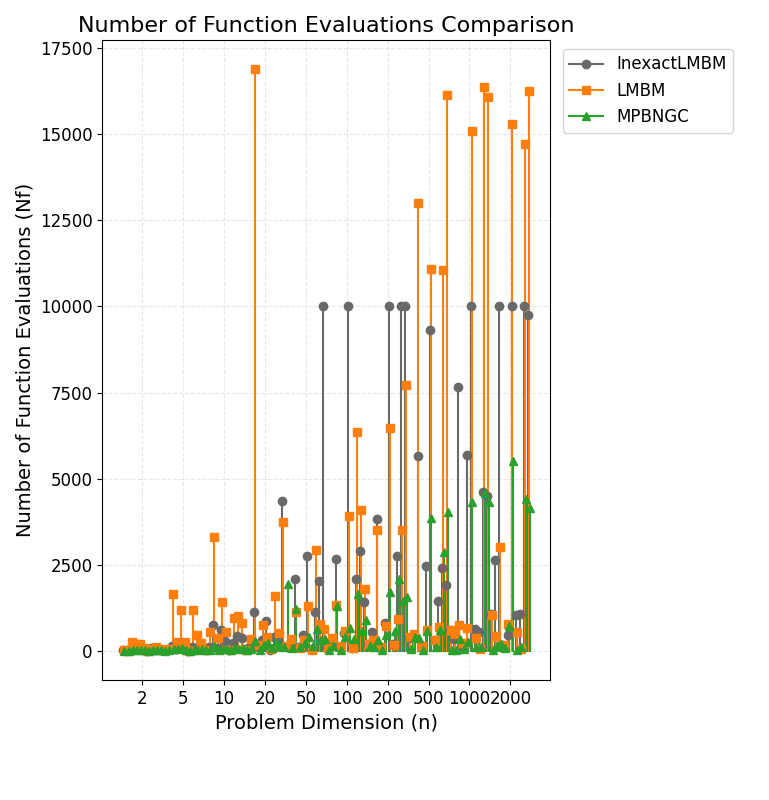}}%\hspace{-1.7cm}
\subfigure[CPU time]{\label{fig:cpu}
\includegraphics[width=0.49\textwidth]{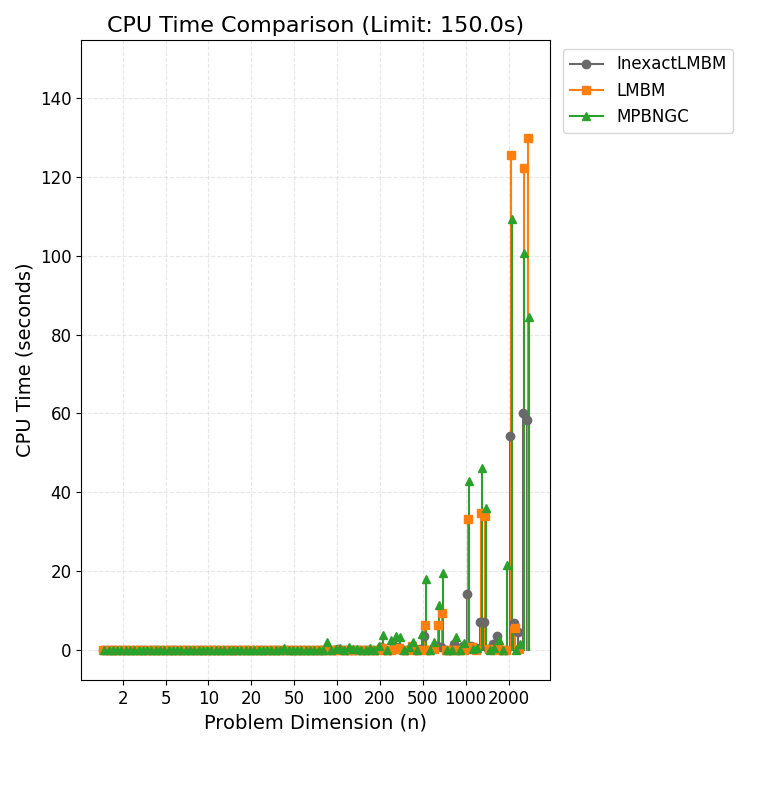}}%\hspace{-0.80cm}
    \caption{Nonsmooth nonconvex problems with exact information:  $n=2,\, 5,\,10,\,20,\,50,\,100,\,200,\,500$, $1000$, and $2000$.    
    {Hollow markers and dashed lines indicate values that lie outside the figure limits.} % Tässä kuvassa näitä ei nyt ole, mutta jos olisi.
    }
    \label{fig:exact_nonconvex}
\end{figure*}

The test problems are nonsmooth and nonconvex, and none of the considered algorithms is guaranteed to converge to a global minimizer. Instead, all solvers are designed to converge to stationary points. Consequently, a solution whose objective value is higher than the global or best-known minimum should not automatically be interpreted as a failure of the algorithm. Such outcomes may correspond to convergence to a different stationary point, which is an expected and theoretically admissible behavior in nonconvex optimization. In particular, multiple stationary points typically exist, and the specific stationary point reached may depend on initialization and algorithmic parameters. The reported relative errors and accuracies with respect to the global minimum therefore reflect the quality of the stationary points found rather than the success or failure of the algorithms in a strict sense.

%In our experiments, an algorithm is considered to fail only if it does not satisfy the prescribed stationarity conditions or termination criteria within the allotted budget.

From Figure~\ref{fig:exact_nonconvex}, four main trends can be observed:
\begin{enumerate}
\item Most notably, {\tt InexactLMBM} more frequently converged to the global solution than either {\tt LMBM} or {\tt MPBNGC} (Figure \ref{fig:error}).
\item The accuracies of {\tt InexactLMBM} and {\tt LMBM} were often higher than asked ($\varepsilon = 10^{-5}$), while {\tt MPBNGC} usually returned results with accuracies close to the asked tolerance (Figure \ref{fig:accuracy}).   
\item The number of function evaluations required by {\tt InexactLMBM} was generally lower than that of {\tt LMBM}. However, in \emph{nine} out of the 100 test problems, {\tt InexactLMBM} terminated only after reaching the maximum number of iterations (Figure \ref{fig:Nf}).
\item The computational times of {\tt InexactLMBM} were less than, or comparable to, those of {\tt LMBM} and {\tt MPBNGC}, even for larger-scale problems (Figure \ref{fig:cpu}).
\end{enumerate}
Although {\tt MPBNGC} typically required fewer function evaluations than both {\tt InexactLMBM} and {\tt LMBM} -- particularly with the new termination criteria introduced in version~4.0 -- it was nevertheless significantly more time-consuming on larger problems. The exception being {\em three} problems with $n=2000$, where  {\tt MPBNGC} was faster than {\tt LMBM} (but not faster than {\tt InexactLMBM}).

One reason for relatively long computational times of {\tt LMBM} in our experiments is that it needs function values and subgradients in separate functions/subroutines. In case of Ferrier polynomials (i.e., half of all problems), this almost doubles the computational burden. Both {\tt InexactLMBM} and {\tt MPBNGC} compute function values and subgradients in one function. Nevertheless, our preliminary tests showed that {\tt InexactLMBM} was usually faster than {\tt LMBM} even if we computed the function values and subgradients separately.

\paragraph*{Problems with inexact information.} Given the strong performance of {\tt InexactLMBM} on problems with exact information, we next analyze its behavior under inexact information by comparing inexact and exact computations. To address the random nature of problems for noise forms $\mathcal{N}1$ -- $\mathcal{N}4$, we made
ten runs with different random noise for each problem with each number of variables. Results are averaged over these ten runs. Noise form $\mathcal{N}0$
 is deterministic, so no repeating is required. Since there is no sense to try tolerance smaller than the added randomness, we used the stopping criterion $w_k < \max\{ \varepsilon, \bar{q}\}$.

Figures \ref{fig:inexact_N1N3} -- \ref{fig:inexact_N2N4} compare {\tt InexactLMBM} with different noise types $\mathcal{N}0$ -- $\mathcal{N}4$ with $\bar{q}=0.01$, and Figures \ref{fig:inexact_N1_accuracy} -- \ref{fig:inexact_N4_accuracy} compare results with different noise levels $\bar{q} = 0, \, 0.0001, \, 0.001,$ and $0.01$ (see also Figures \ref{fig:inexact_comparison_different_noises} and \ref{fig:inexact_noise}
 in Appendix~\ref{appendix_results}). In the figure legends, noisy variants are labeled {\tt InexactLMBM\_N}$i$ (e.g., {\tt InexactLMBM\_N1}), where the suffix denotes the noise type; the exact/noiseless variant ($\mathcal{N}0$) appears as {\tt inexactLMBM} without a suffix. In addition, the noise bound $\bar{q}$ may be added to suffix when needed.
%\todo{Näistä kuvista ehkä mielenkiintoisimpia on tuo accuracy ja Nf} 

From Figure \ref{fig:accuracy_n1n3} we see that for $n\geq 10$, {\tt InexactLMBM} (without noise) typically attains higher accuracy than {\tt InexactLMBM\_N1} and {\tt InexactLMBM\_N3}. In addition, {\tt InexactLMBM\_N3} usually outperforms {\tt InexactLMBM\_N1}.  This is what we would expect, since $\mathcal{N}3$ 
perturbs only subgradients, whereas $\mathcal{N}1$ also perturbs function values. In contrast, for $n < 10$, runs with {\tt InexactLMBM\_N1} often appear unusually accurate. This is largely a noise artifact -- because the noise added to function values can be negative, the reported objective may be artificially lowered, creating a misleading impression of accuracy. This effect is absent with 
$\mathcal{N}3$, which does not perturb function values. Overall, the attained accuracy under noise was better than expected for {\tt InexactLMBM\_N3} -- final errors were often below the noise bound $\bar{q}$ -- while {\tt InexactLMBM\_N1} it was slightly worse than anticipated. Nevertheless, 
comparing {\tt InexactLMBM\_N1} with the original {\tt LMBM} (Figure \ref{fig:accuracy}), we observe that both methods tend to miss the global minimum on the same problem instances, which, as pointed out earlier, can not be considered as failure with local search methods.

Figure \ref{fig:Nf_n1n3} shows that the noisy variants, {\tt InexactLMBM\_N1}  and {\tt InexactLMBM\_N3}, often require fewer function evaluations than {\tt InexactLMBM} with exact computations. 
This is partly due to the effectively looser accuracy demand under noise, and partly because the injected perturbations can facilitate progress in slowly convergent or flat regions.

Figure \ref{fig:inexact_N2N4}, which considers vanishing-noise problems, shows trends similar to Figure \ref{fig:inexact_N1N3}, with {\tt Inexact\-LMBM\_N1} and {\tt InexactLMBM\_N3} replaced by their vanishing-noise counterparts {\tt InexactLMBM\_N2} and {\tt InexactLMBM\_N4}, respectively.
The main differences are: (i) the misleading over-accuracy for small $n$ seen with {\tt InexactLMBM\_N1} disappears, and (ii) {\tt InexactLMBM\_N2} is overall more accurate than {\tt InexactLMBM\_N1} (see also Figure \ref{fig:inexact_comparison_different_noises} in Appendix~\ref{appendix_results}). By contrast, {\tt InexactLMBM\_N3} and {\tt InexactLMBM\_N4} show very similar accuracy. In particular, both {\tt InexactLMBM\_N3} and {\tt InexactLMBM\_N4} achieve accuracy at or above the desired accuracy level (noise bound) for problems with up to 1000 variables. Accordingly, the absence of a marked accuracy advantage is consistent with the attainable accuracy dictated by the noise level and termination criteria.

\begin{figure*}[p] 
\centering
%\subfigure[Relative error]{\label{fig:error_n1n3}
%\includegraphics[width=0.49\textwidth]%{figures/N1N3_f_comparison.png}}%\hspace{-1.7cm}
\subfigure[Accuracy]{\label{fig:accuracy_n1n3}
\includegraphics[width=0.49\textwidth]{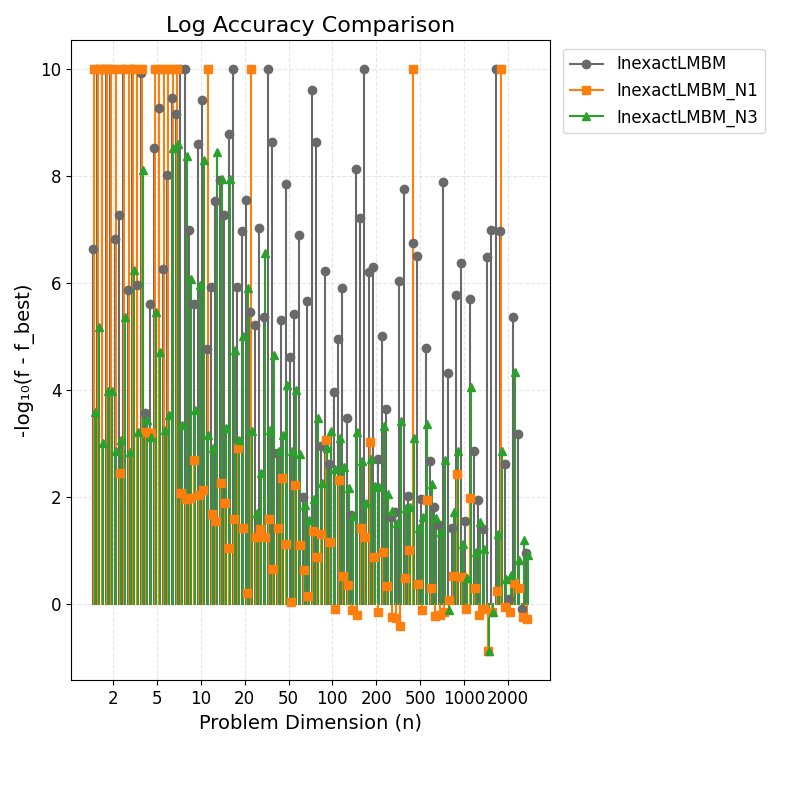}}%\hspace{-1.7cm}
\subfigure[Function evaluations]{\label{fig:Nf_n1n3}
\includegraphics[width=0.49\textwidth]{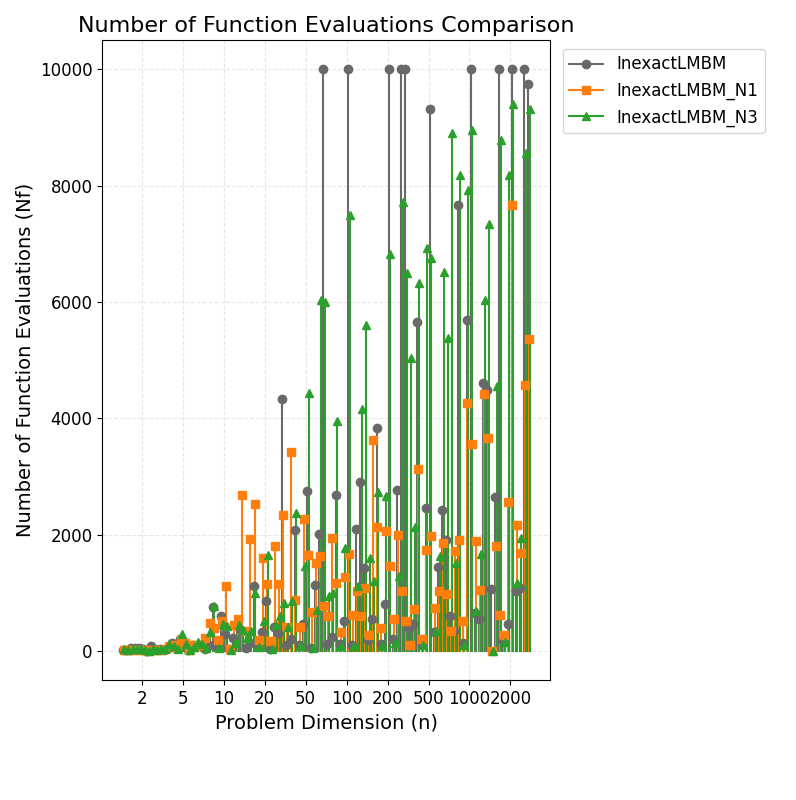}}%\hspace{-1.7cm}
%\subfigure[CPU time]{\label{fig:cpu_n1n3}
%\includegraphics[width=0.49\textwidth]%{figures/N1N3_cpu_comparison.png}}%\hspace{-0.80cm}
    \caption{Nonsmooth nonconvex problems with constant noises {$\mathcal{N}1$ and $\mathcal{N}3$, the noise bound $\bar{q}=0.01$}, and $n=2,\, 5,\,10,\,20,\,50,\,100,\,200,\,500$, $1000$, and $2000$.    
    % {Hollow markers and dashed lines indicate values that lie outside the figure limits.} % Tässä kuvassa näitä ei nyt ole, mutta jos olisi.
    }
    \label{fig:inexact_N1N3}
\end{figure*}

\begin{figure*}[p] 
\centering
%\subfigure[Relative error]{\label{fig:error_n2n4}
%\includegraphics[width=0.49\textwidth]{figures/N2N4_f_comparison.png}}%\hspace{-1.7cm}
\subfigure[Accuracy]{\label{fig:accuracy_n2n4}
\includegraphics[width=0.49\textwidth]{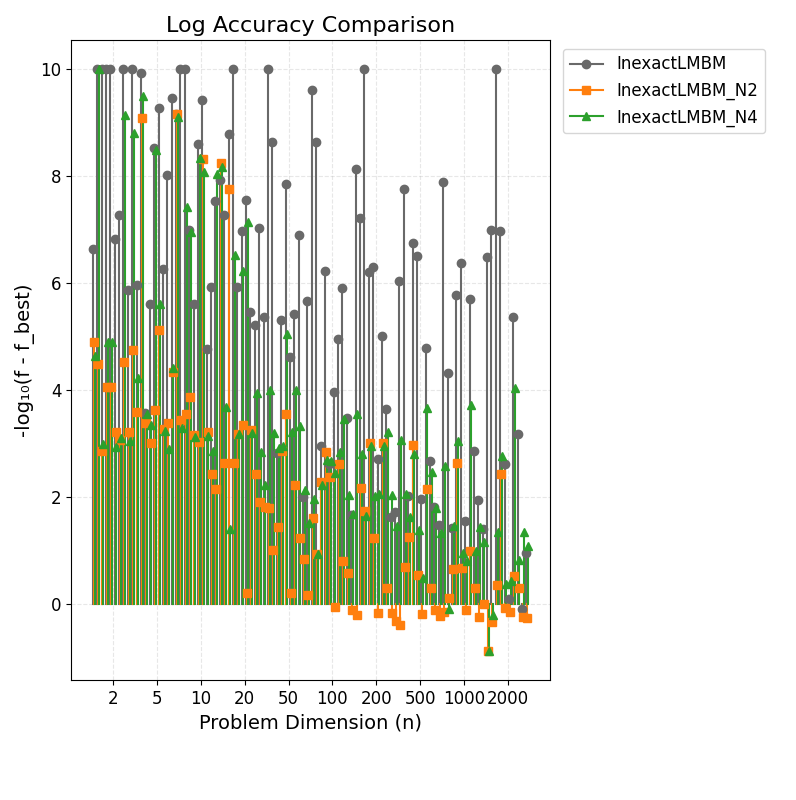}}%\hspace{-1.7cm}
\subfigure[Function evaluations]{\label{fig:Nf_n2n4}
\includegraphics[width=0.49\textwidth]{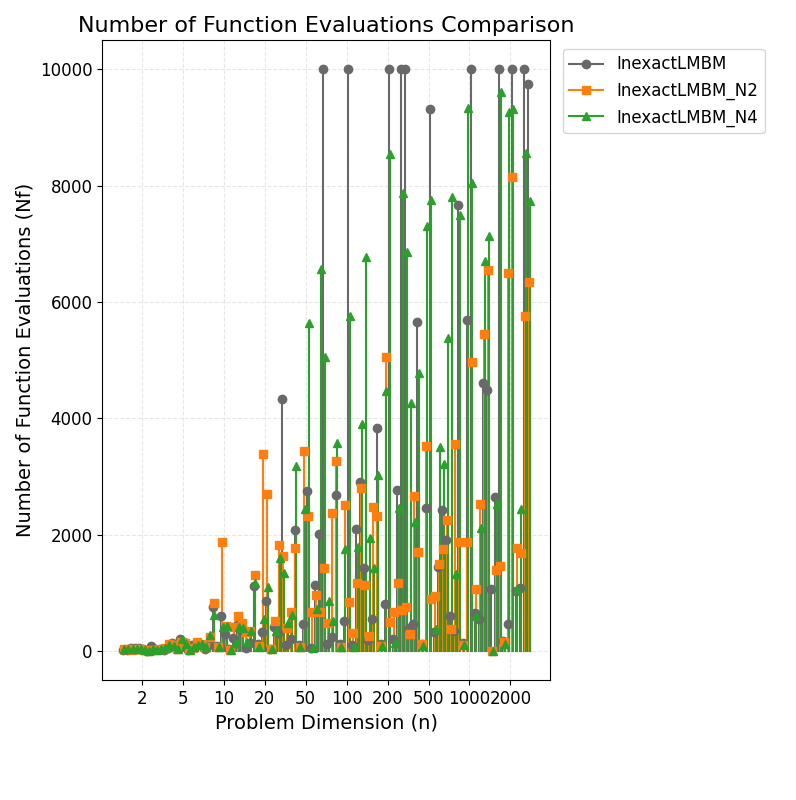}}%\hspace{-1.7cm}
%\subfigure[CPU time]{\label{fig:cpu_n2n4} 
%\includegraphics[width=0.49\textwidth]{figures/N2N4_cpu_comparison.png}}%\hspace{-0.80cm}
    \caption{Nonsmooth nonconvex problems with vanishing noises {$\mathcal{N}2$ and $\mathcal{N}4$, the noise bound $\bar{q}=0.01$}, and  $n=2,\, 5,\,10,\,20,\,50,\,100,\,200,\,500$, $1000$, and $2000$.    
    %{Hollow markers and dashed lines indicate values that lie outside the figure limits.} % Tässä kuvassa näitä ei nyt ole, mutta jos olisi.
    }
    \label{fig:inexact_N2N4}
\end{figure*}

Finally, Figures \ref{fig:inexact_N1_accuracy}--\ref{fig:inexact_N4_accuracy} and Table \ref{deviations} compare performance across different noise bounds $\bar{q} = 0, \, 0.0001$, $0.001,$ and $0.01$
within each noise type 
$\mathcal{N}1$ -- $\mathcal{N}4$ (see also Figures \ref{fig:inexact_comparison_different_noises} and \ref{fig:inexact_noise} in Appendix~\ref{appendix_results}). Across all noise types, smaller $\bar{q}$  consistently yields higher accuracy. Beyond this expected dependence on the noise level, the figures are consistent with the results discussed above.  Complementing these figures, Table \ref{deviations} reports the mean standard deviations of the function values for $\mathcal{N}1$ -- $\mathcal{N}4$ at positive noise levels $\bar{q} = 0.0001, \, 0.001,$ and $0.01$. For smaller problems ($n<100$), 
lower noise consistently yields smaller deviations. For larger problems, however, this monotonic behavior can break down: deviations do not always decrease as the noise level is reduced. A plausible explanation is dimensional amplification -- perturbations accumulate with dimension -- consistent with the generally larger deviations observed as $n$ increases. We also note a change in termination behavior: for smaller problems, runs typically stop by meeting the accuracy tolerance (in noisy settings, the effective tolerance scales with the specified noise level), whereas for larger problems termination is more often triggered by stagnation of the function value or the subgradient norm.

\begin{figure*} 
\centering
\includegraphics[width=0.9\textwidth]{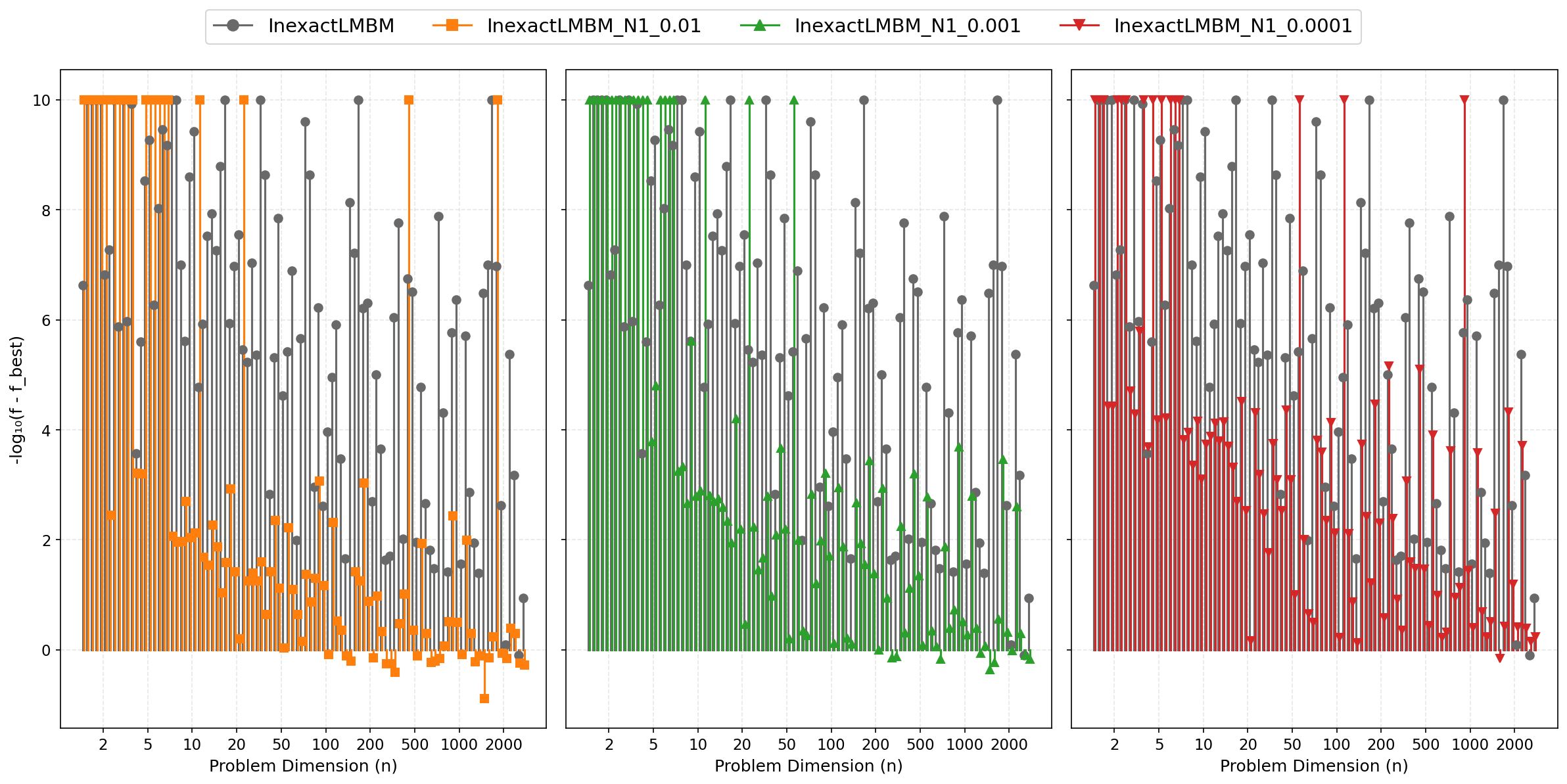}%\hspace{-1.7cm}
    \caption{Comparison of different noise bounds $\bar{q} = 0,\, 0.0001, \,0.001$, and $0.01$ with noise type $\mathcal{N}1$ (constant noise in both function values and subgradients), and $n=2,\, 5,\,10,\,20,\,50,\,100,\,200,\,500$, $1000$, and $2000$.    
    }
    \label{fig:inexact_N1_accuracy}
\end{figure*}
\begin{figure*} 
\centering
\includegraphics[width=0.9\textwidth]{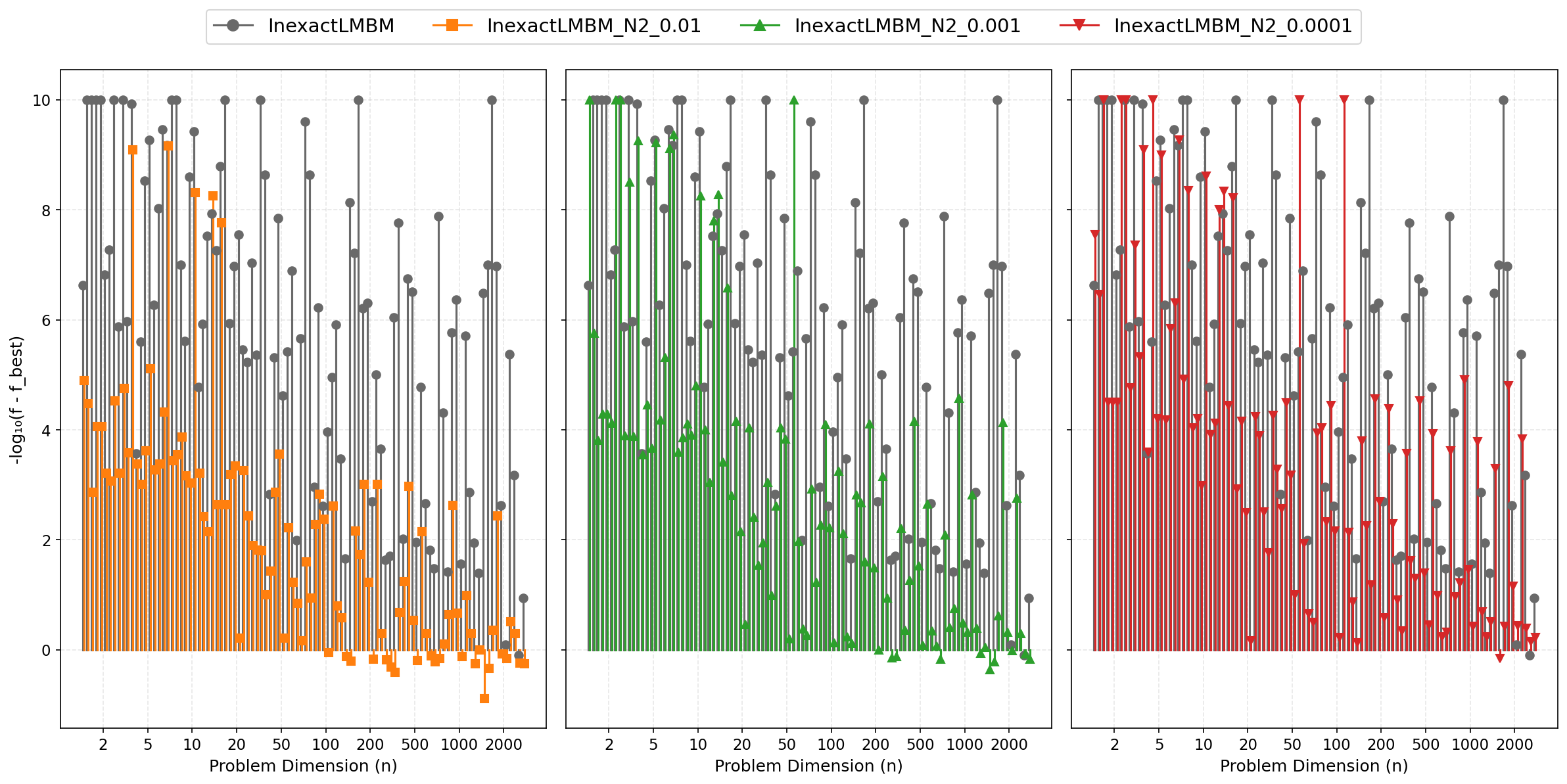}%\hspace{-1.7cm}
    \caption{Comparison of different noise bounds $\bar{q} = 0, \, 0.0001, \,0.001$, and $0.01$ with noise type $\mathcal{N}2$ (vanishing noise in both function values and subgradients), and $n=2,\, 5,\,10,\,20,\,50,\,100,\,200,\,500$, $1000$, and $2000$.    
    }
    \label{fig:inexact_N2_accuracy}
\end{figure*}
\begin{figure*} 
\centering
\includegraphics[width=0.9\textwidth]{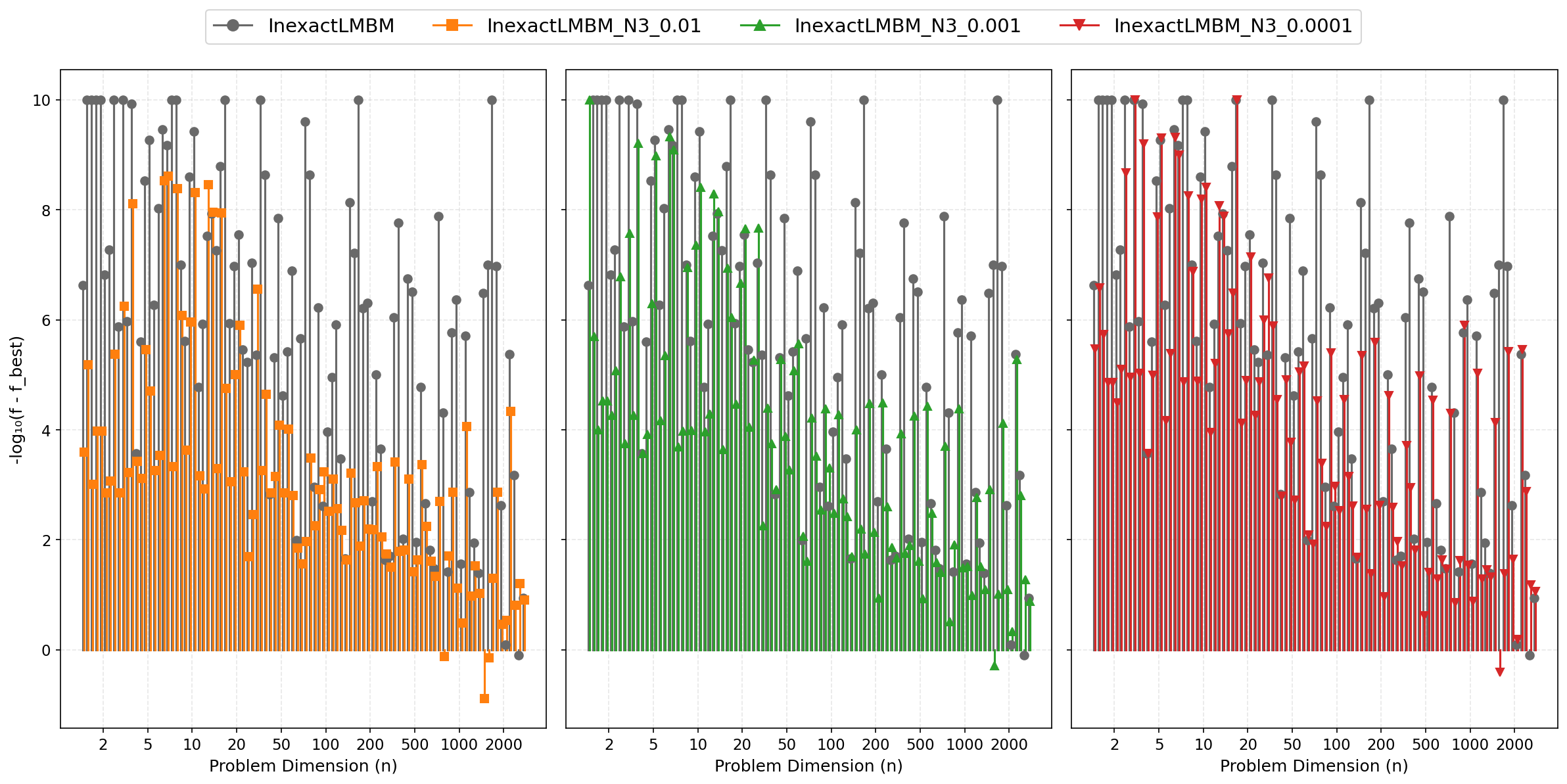}%\hspace{-1.7cm}
    \caption{Comparison of different noise bounds $\bar{q} = 0, \, 0.0001, \,0.001$, and $0.01$ with noise type $\mathcal{N}3$ (constant noise in subgradients), and $n=2,\, 5,\,10,\,20,\,50,\,100,\,200,\,500$, $1000$, and $2000$.    
    }
    \label{fig:inexact_N3_accuracy}
\end{figure*}
\begin{figure*} 
\centering
\includegraphics[width=0.9\textwidth]{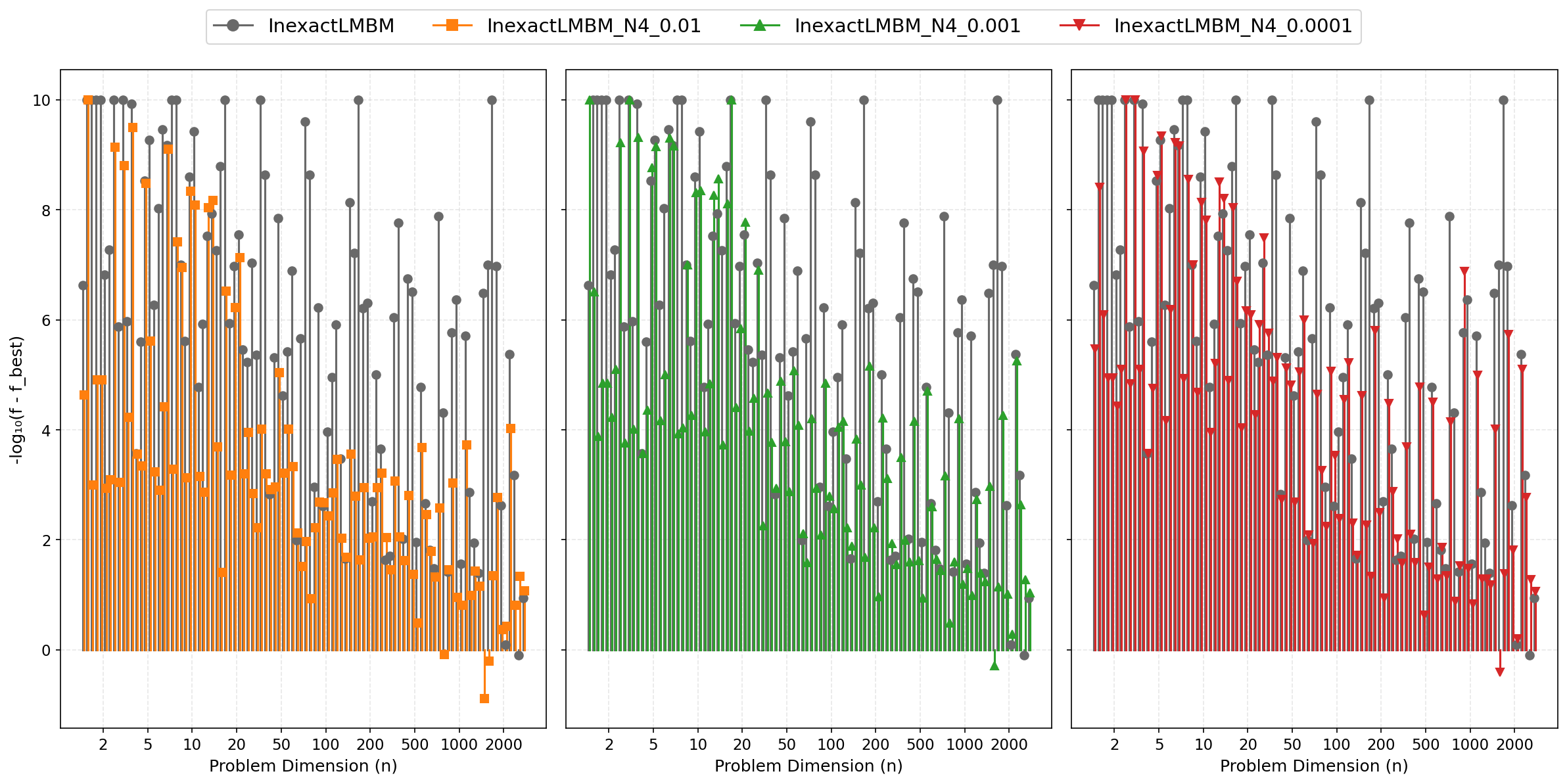}%\hspace{-1.7cm}
    \caption{Comparison of different noise bounds $\bar{q} = 0, \, 0.0001, \,0.001$, and $0.01$ with noise type $\mathcal{N}4$ (vanishing noise in subgradients), and $n=2,\, 5,\,10,\,20,\,50,\,100,\,200,\,500$, $1000$, and $2000$.    
    }
    \label{fig:inexact_N4_accuracy}
\end{figure*}

\begin{table}[h!]% Tässä korjatut arvot ok.
  \caption{\label{deviations} Average standard deviations in function values with different noise levels}
    \resizebox{1.0\textwidth}{!}{
 %\begin{adjustbox}{max width=1.3\textwidth}
 \begin{tabular}{l llll llll llll llll } 
\hline\noalign{\smallskip}
	&&	$\mathcal{N}1$	&&		&		&	$\mathcal{N}2$	&&		&		&	$\mathcal{N}3$	&&		&		&	$\mathcal{N}4$	&		
             \\ 
             \cmidrule{2-4} \cmidrule{6-8} \cmidrule{10-12} \cmidrule{14-16}
  $n/ \bar{q}$ &	0.01	&	0.001	&	0.0001	&&	0.01	&	0.001	&	0.0001	&&	0.01	&	0.001	&	0.0001	&&	0.01	&	0.001	&	0.0001\\ 
\noalign{\smallskip}\hline\noalign{\smallskip}
2	&	4.58E-03	&	3.56E-04	&	5.87E-05	&&	3.61E-04	&	5.49E-05	&	2.63E-05	&&	3.75E-04	&	3.16E-05	&	3.87E-06	&&	1.78E-04	&	2.65E-05	&	1.52E-06\\
5	&	5.25E-03	&	5.80E-04	&	4.96E-05	&&	2.88E-04	&	7.82E-05	&	2.09E-05	&&	2.19E-04	&	2.43E-05	&	3.10E-06	&&	2.22E-04	&	2.03E-05	&	1.45E-06\\
10	&	1.57E-02	&	1.10E-03	&	3.11E-04	&&	3.20E-03	&	2.38E-04	&	2.36E-04	&&	2.49E-04	&	4.48E-05	&	3.20E-06	&&	3.64E-04	&	3.18E-05	&	4.00E-06\\
20	&	1.02E-01	&	6.78E-02	&	6.48E-02	&&	7.28E-02	&	6.59E-02	&	6.45E-02	&&	6.83E-03	&	1.76E-03	&	1.09E-05	&&	1.49E-02	&	1.78E-03	&	3.49E-06\\
50	&	2.29E-01	&	2.25E-01	&	1.58E-01	&&	2.28E-01	&	2.25E-01	&	1.57E-01	&&	4.86E-03	&	2.89E-03	&	2.13E-03	&&	3.55E-03	&	2.67E-03	&	2.16E-03\\
100	&	1.94E-01	&	2.13E-01	&	1.65E-01	&&	2.28E-01	&	2.16E-01	&	1.65E-01	&&	7.10E-03	&	2.86E-03	&	3.10E-03	&&	3.20E-02	&	3.48E-03	&	3.51E-03\\
200	&	4.65E-01	&	2.29E-01	&	1.60E-01	&&	3.92E-01	&	2.28E-01	&	1.59E-01	&&	7.24E-03	&	3.89E-02	&	3.96E-02	&&	6.73E-03	&	3.74E-02	&	4.02E-02\\
500	&	6.68E-01	&	3.32E-01	&	2.15E-01	&&	6.40E-01	&	3.31E-01	&	2.18E-01	&&	1.27E-02	&	4.13E-02	&	9.37E-02	&&	5.82E-02	&	4.21E-02	&	9.63E-02\\
1000	&	6.19E-01	&	3.64E-01	&	2.29E-01	&&	6.10E-01	&	3.64E-01	&	2.28E-01	&&	2.05E-01	&	1.25E-01	&	7.15E-02	&&	1.81E-01	&	1.23E-01	&	7.68E-02\\
2000	&	4.58E-01	&	8.04E-01	&	4.28E-01	&&	4.61E-01	&	8.22E-01	&	4.28E-01	&&	2.74E-01	&	2.24E-01	&	1.50E-01	&&	2.70E-01	&	2.08E-01	&	1.47E-01\\
\noalign{\smallskip}\hline\noalign{\smallskip}
\end{tabular}
%\end{adjustbox}
}
\end{table}

%{\color{red} 
%Pitääkö meillä olla testejä konveksifiointiparametrin $\eta$ suuruudesta. Warren Harella on. Nopea testaus sanoo, että $\mathcal{N}3$ ja $\mathcal{N}4$ tuo parametri on hyvinkin bounded, mutta $\mathcal{N}1$ ja $\mathcal{N}2$ kasvaa välillä tosi suureksi mikä taas käytännössa aiheuttaa, että $\beta = 0$. Nopea testaus $\mathcal{N}1$, jossa $\beta = \max(\beta, |\alpha|)$ ei kuitenkaan parantanut tuloksia (jos ei huonontanutkaan).
%}

These experiments indicate that {\inexactLMBM} is efficient for large-scale nonsmooth optimization and a strong alternative not only with inexact information but even when exact evaluations are available: it converged to a global minimum more often than the other solvers tested, achieved at least comparable accuracy, and required less computational time.
In inexact settings, when both function values and subgradients are perturbed ($\mathcal{N}1,\mathcal{N}2$), accuracy naturally degrades with noise. Nevertheless, for relatively small problems ($n \leq 50$) and small bounds ($\bar{q}\leq0.001$), the method still attains the desired accuracy, especially in the vanishing-noise variant ($\mathcal{N}2$). When only subgradients are perturbed ($\mathcal{N}3/\mathcal{N}4$), the attained accuracy is typically commensurate with -- or better than -- the nominal noise level, and with small bounds ($\bar{q}\leq0.001$) it is often nearly indistinguishable from the noiseless case.

\section{Conclusions}
\label{sec_Conclusion}
This paper introduces \inexactLMBM, a novel inexact limited memory bundle method for large-scale nonsmooth and nonconvex optimization that explicitly accommodates inexact function values and/or subgradients. Such inexactness may arise in practice, for instance, from measurement or modeling error, numerical approximations, stochastic simulations, and privacy-preserving perturbations (e.g., differentially private noise). We have proved the global convergence of the proposed method to an approximately stationary point under the standard assumption that the objective function is locally Lipschitz continuous. In contrast to traditional approaches, however, our analysis does not require an additional semi-smoothness assumption, which makes {\tt InexactLMBM} applicable to a broader class of nonsmooth problems.

The performance of {\tt InexactLMBM} was evaluated across several scenarios. In the case of exact function and subgradient information, it was benchmarked against the original {\tt LMBM} and the proximal bundle method {\tt MPBNGC}. The results indicate that {\tt InexactLMBM} more consistently reached high-quality solutions while requiring fewer function evaluations and less computational time. Notably, although all tested methods are local optimization methods and therefore not expected to converge to a global solution, {\inexactLMBM} reached the global solution more often than the competing methods, which may be viewed as an additional practical advantage. To study the effects of inexactness, the algorithm is tested with both constant or decreasing noise in the subgradients, as well as in scenarios where both 
function values and subgradients are available only inexactly. In all cases, the results demonstrate that {\inexactLMBM} remains robust and efficient for large-scale nonsmooth optimization despite the presence of noise.

An important direction for future work is to explore the use of {\inexactLMBM} as a privacy preserving optimization mechanism in differentially private machine learning (analogously to DP-SGD \cite{abadi2016deep}). In this setting, the noise added to ensure privacy naturally gives rise to inexact function and subgradient information, which aligns well with the proposed framework. This opens an interesting avenue for further theoretical analysis and practical algorithm development.

\section*{Acknowledgements}
This work was financially supported by the Research Council of Finland (projects no.\ \#340182 and \#340140), Jenny and Antti Wihuri Foundation, University of Turku Graduate School UTUGS -- Doctoral Programme in Exact Sciences (EXACTUS), University of Turku, and the European Union's Horizon Europe project Privacy Preserving Identity Management for Digital Wallet and Secure Data Sharing and Processing for Cyber Threat Intelligence Data (PRIVIDEMA, Grant Agreement No.\ 101167964).
Views and opinions expressed are however those of the authors only and do not necessarily reflect those of the European Union or the European Cybersecurity Competence Centre. Neither the European Union nor the European Cybersecurity Competence Centre can be held responsible for them.

\section*{Disclosure statement}
The authors report there are no competing interests to declare.

\section*{Data availability statement}
The data that support the findings of this study are available from the corresponding author, upon reasonable request.

\bibliographystyle{tfs}
\bibliography{references_uusi}

\newpage
\appendix
\section{Test problems}
\label{appendix_testproblems}

In our numerical experiments we used \emph{five} nonsmooth nonconvex academic minimization problems from \cite{HaaMieMak:2004} and \emph{five} Ferrier polynomials from \cite{HarSagSol:2016}. Below find information of these problems. 

\paragraph*{Large-scale nonsmooth nonconvex test problems from \cite{HaaMieMak:2004}.}
We consider the following non\-smooth, generally nonconvex, objectives with starting point given as $\x^{(1)}$:
% Vaihda järjestys, niin numerot tulee oikein, mutta silloin melkein pitäis vaihtaa ne tuloksissakin
%{\bf 6. Number of Active Faces} \> \> \> \> \\
% Optimum x^* = 0, 0, ....
\begin{align*}
 &f_1(\x)= \max_{1 \leq i \leq n}\left\{\, g\left(- \sum_{i=1}^n
    x_i\right),g(x_i) \,\right\}, \\
& \qquad \qquad \text{where } g(y) = \ln\left(|y|+1\right) \text{ and }
x_i^{(1)}=1 \text{ for all } i=1,\ldots,n. \\
%
%{\bf 7. Nonsmooth generalization of Brown function 2} \>  
% Optimum x^* = 0, 0, ....
&f_2(\x)= \sum_{i=1}^{n-1}\left( \left| x_i \right|^{x_{i+1}^2+1}+\left|
    x_{i+1} \right|^{x_i^2+1} \right), \\
& \qquad \qquad \text{where } x_i^{(1)}=\begin{cases}
     -1, \quad \text{when} \mod(i,2)=1,\\
     \phantom{-}1,  \quad \text{when} \mod(i,2)=0.
 \end{cases}\\
%
%{\bf 8. Chained Mifflin 2} \> \> \> \> \\
% Optimum x^* = not available
 &f_3(\x)= \sum_{i=1}^{n-1} \left(\,
  -x_i+2\left(x_i^2+x_{i+1}^2-1\right)+1.75\left|x_i^2+x_{i+1}^2-1\right|\,\right), \\
& \qquad \qquad \text{where } x_i^{(1)}=-1 \text{ for all } i=1,\ldots,n. \\
%
%{\bf 9. Chained Crescent I} \> \> \> \> \\
 &f_4(\x)= \max \left\{\, \sum_{i=1}^{n-1} \left(x_i^2 +\left(
    x_{i+1}-1\right)^2+x_{i+1}-1\right),\right.
 \left.\,\sum_{i=1}^{n-1} \left(-x_i^2 -\left( x_{i+1}-1\right)^2+x_{i+1}+1\right) \right\},\\
& \qquad \qquad \text{where } x_i^{(1)}=\begin{cases}
     -1.5, \quad \text{when} \mod(i,2)=1,\\
     \phantom{-}2.0,  \quad \text{when} \mod(i,2)=0.
 \end{cases}\\
%
%{\bf 10. Chained Crescent II} \> \> \> \> \\
&f_5(\x)= \sum_{i=1}^{n-1} \max \left\{x_i^2 +\left(
    x_{i+1}-1\right)^2+x_{i+1}-1, -x_i^2 -\left( x_{i+1}-1\right)^2+x_{i+1}+1\right\},
 \\
& \qquad \qquad \text{where } x_i^{(1)}=\begin{cases}
     -1.5, \quad \text{when} \mod(i,2)=1,\\
     \phantom{-}2.0,  \quad \text{when} \mod(i,2)=0.
 \end{cases}\\
\end{align*}

All these functions but $f_3$ have $\x^*= \boldsymbol{0}$ as a global minimizer.
For $f_3$ the best known solutions are 
$-1$ for $n=2$,
$-2.98 $ for $n=5$,
$-6.51 $ for $n=10$,
$-13.58 $ for $n=20$,
$-34.80 $ for $n=50$,
$-70.15 $ for $n=100$,
$-140.86 $ for $n=200$,
$-352.99 $ for $n=500$,
$-706.54 $ for $n=1000$, and
$-1413.65$ for $n=2000$. In our experiments, we have used $\x^*$ that correspond to these function values for computations of vanishing noises $\mathcal{N}2$ and $\mathcal{N}4$.

\paragraph*{Ferrier polynomials from \cite{HarSagSol:2016}:} For $\x \in \R^n$ and $i = 1, \ldots, n$, define
$$h_i(\x) = (i x_i^2 - 2 x_i) + \sum_{j=1}^n x_j.$$
Using these building blocks, we consider the following nonsmooth, generally nonconvex objectives:
\begin{align*}
    &f_6(\x) = \sum_{i=1}^n |h_i(\x)|, \\
    &f_7(\x) = \sum_{i=1}^n \big(h_i(\x)\big)^2, \\
    &f_8(\x) = \max_{1 \le i \le n} |h_i(\x)|,\\
    &f_9(\x) = \sum_{i=1}^n |h_i(\x)|+ \frac{1}{2} |\x|^2, \qquad \text{and} \\
    &f_{10}(\x) = \sum_{i=1}^n |h_i(\x)|+ \frac{1}{2} |\x|. \\
\end{align*}
 All these functions $f_5$--$f_{10}$ have $\x^*= \boldsymbol{0}$ as a global minimizer. We use $\x^{(1)} = [1,\, 1/4,\, 1/9,\ldots, 1/n^2]$ as a starting point for optimization.
 
 %For these polynomials, the convexified objective $f_i(\x) + \tfrac{\eta}{2}|x|^2$ is convex whenever $\eta > 2n$, a fact sometimes used in analyses based on convexification.
\newpage

\section{Numerical results}
\label{appendix_results}

\begin{figure*}[h!] 
\centering
\subfigure[{\tt InexactLMBM} vs.\ {\tt LMBM}]{\label{fig:IvsLMBM}
\includegraphics[width=0.49\textwidth]{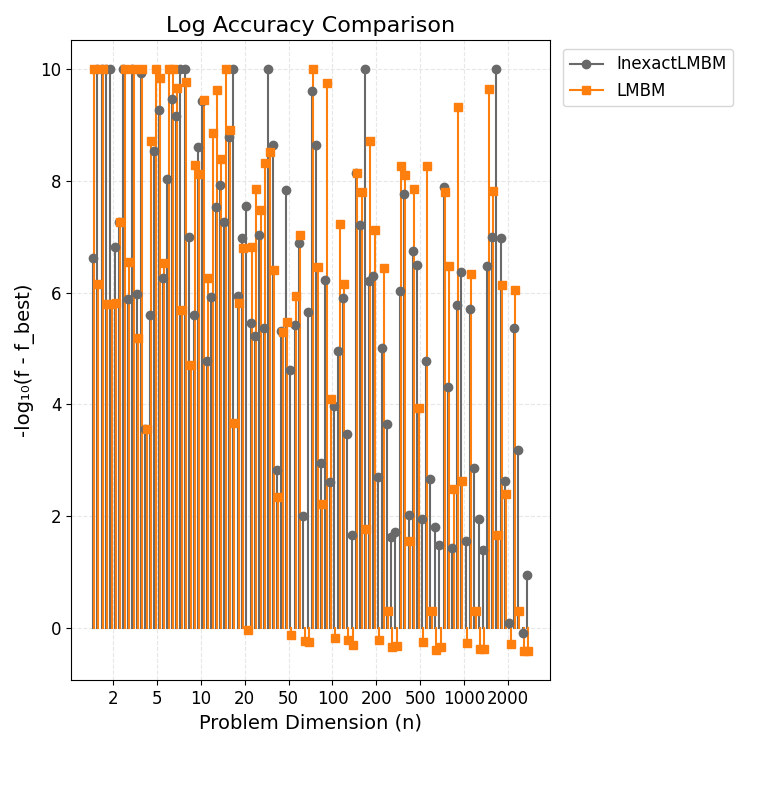}}%\hspace{-1.7cm}
\subfigure[{\tt InexactLMBM} vs.\ {\tt MPBNGC}]{\label{fig:IvsMPBNGC}
\includegraphics[width=0.49\textwidth]{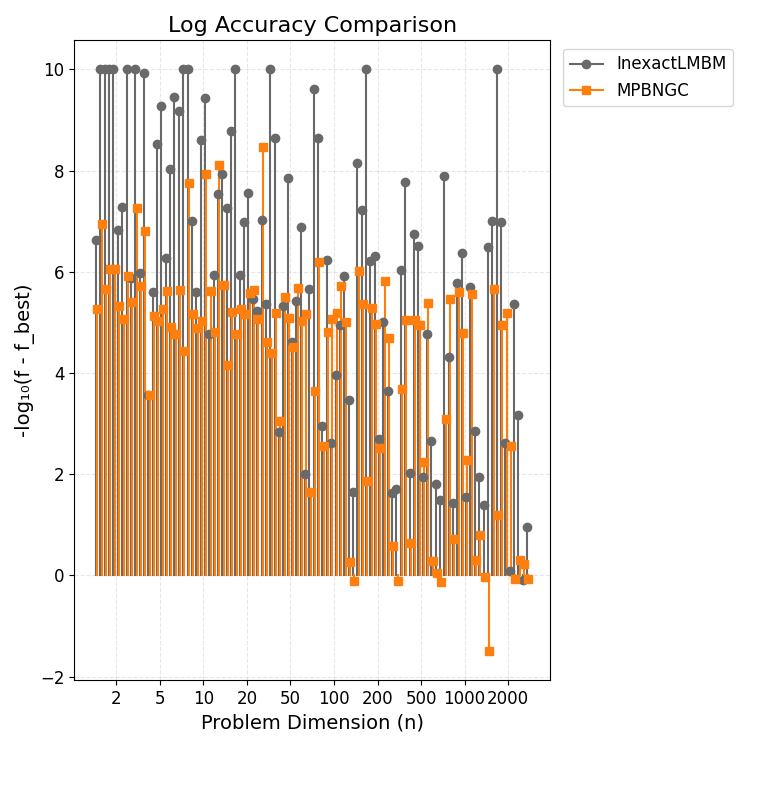}}%\hspace{-1.7cm}
    \caption{Pairwise comparison of {\tt InexactLMBM} with (a) {\tt LMBM}, (b) {\tt MPBNGC}. Nonsmooth nonconvex problems with exact information:  $n=2,\, 5,\,10,\,20,\,50,\,100,\,200,\,500$, $1000$, and $2000$.   
    } % Nämä kuvat on ok
    \label{fig:exact_comparison_pairwise}
\end{figure*}

\begin{figure*}[h!] 
\centering
\subfigure[{$\mathcal{N}1$ vs.\ $\mathcal{N}2$}]{\label{fig:n1}
\includegraphics[width=0.49\textwidth]{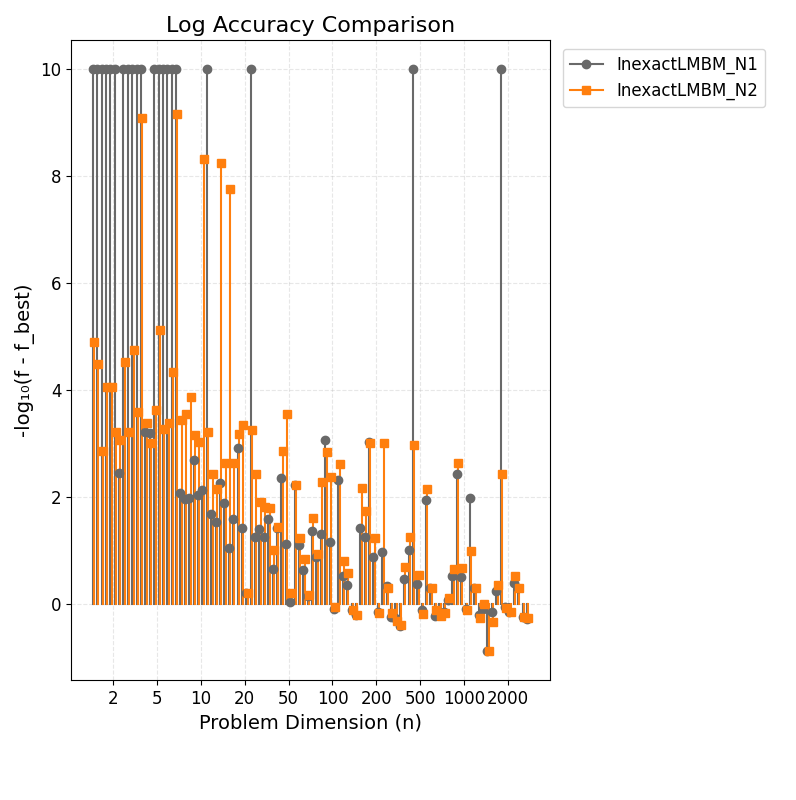}}%\hspace{-1.7cm}
\subfigure[{$\mathcal{N}3$ vs.\ $\mathcal{N}4$}]{\label{fig:n2}
\includegraphics[width=0.49\textwidth]{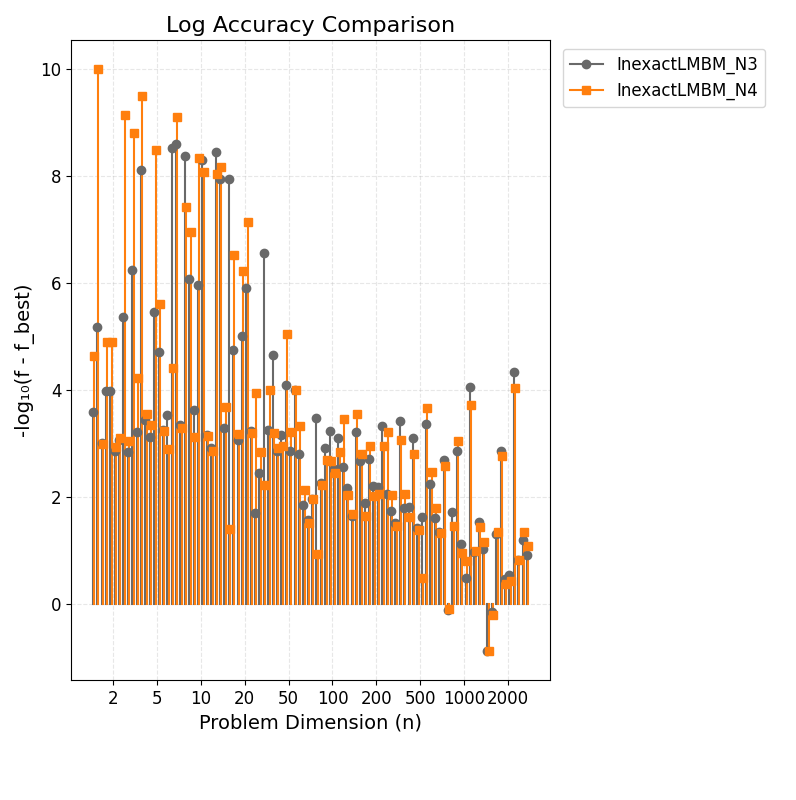}}%\hspace{-1.7cm}

%\subfigure[{\color{red}$\mathcal{N}3$, $q=0$ and $0.01$}]{\label{fig:n3}
%\includegraphics[width=0.49\textwidth]{figures/N3_log_f_comparison0.01.png}}%\hspace{-1.7cm}
%\subfigure[{\color{red}$\mathcal{N}3$, $q=0$ and $0.0001$}]{\label{fig:cpu_n2n4}
%\includegraphics[width=0.49\textwidth]%{figures/N3_log_f_comparison0.0001.png}}%\hspace{-0.80cm}
    \caption{Comparison of different noise types: constant vs.\ vanishing noise (a) in both function values and subgradients, (b) only in subgradients. The noise bound $\bar{q}=0.01$.
    } % Nämä kuvat on ok
    \label{fig:inexact_comparison_different_noises}
\end{figure*}

\begin{figure*} 
\centering
\subfigure[{$\mathcal{N}1$}]{\label{fig:n1}
\includegraphics[width=0.49\textwidth]{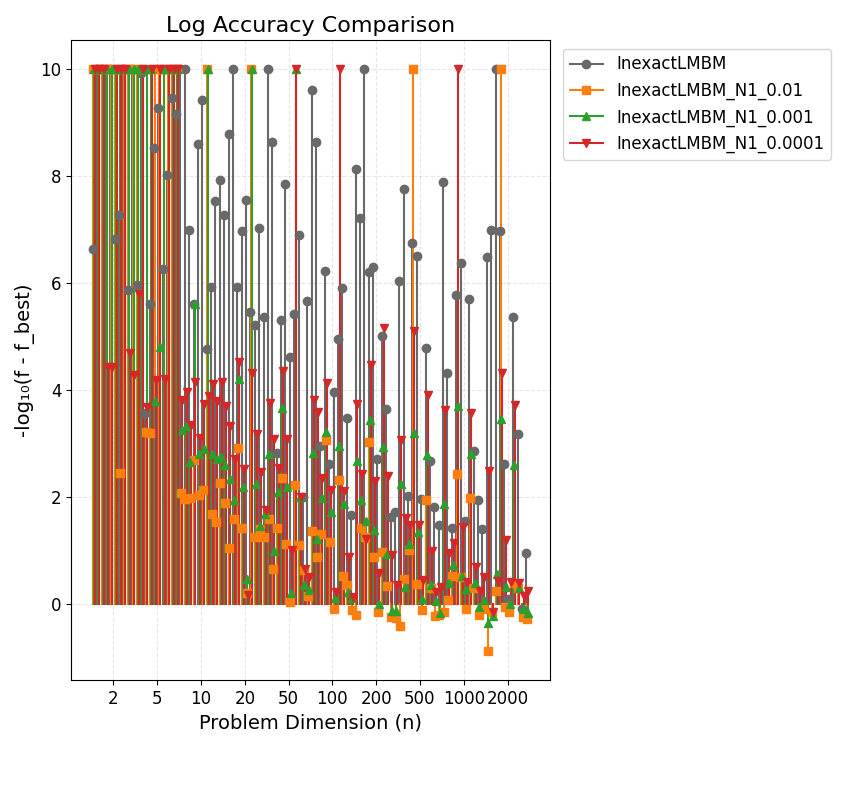}}%\hspace{-1.7cm}
\subfigure[{$\mathcal{N}2$}]{\label{fig:n2}
\includegraphics[width=0.49\textwidth]{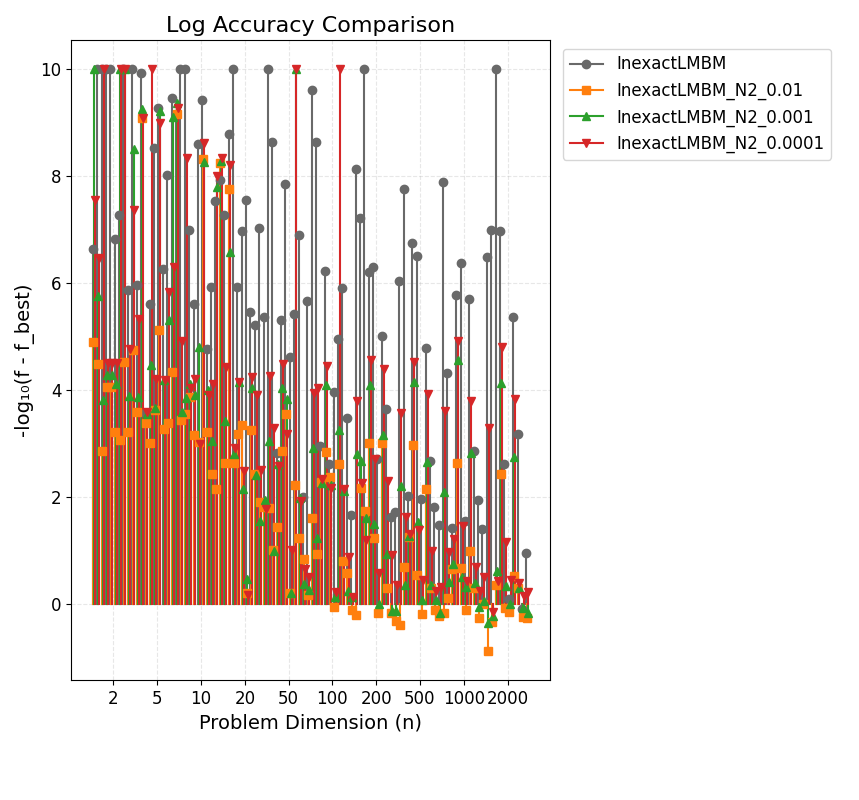}}%\hspace{-1.7cm}

\subfigure[{$\mathcal{N}3$}]{\label{fig:n3}
\includegraphics[width=0.49\textwidth]{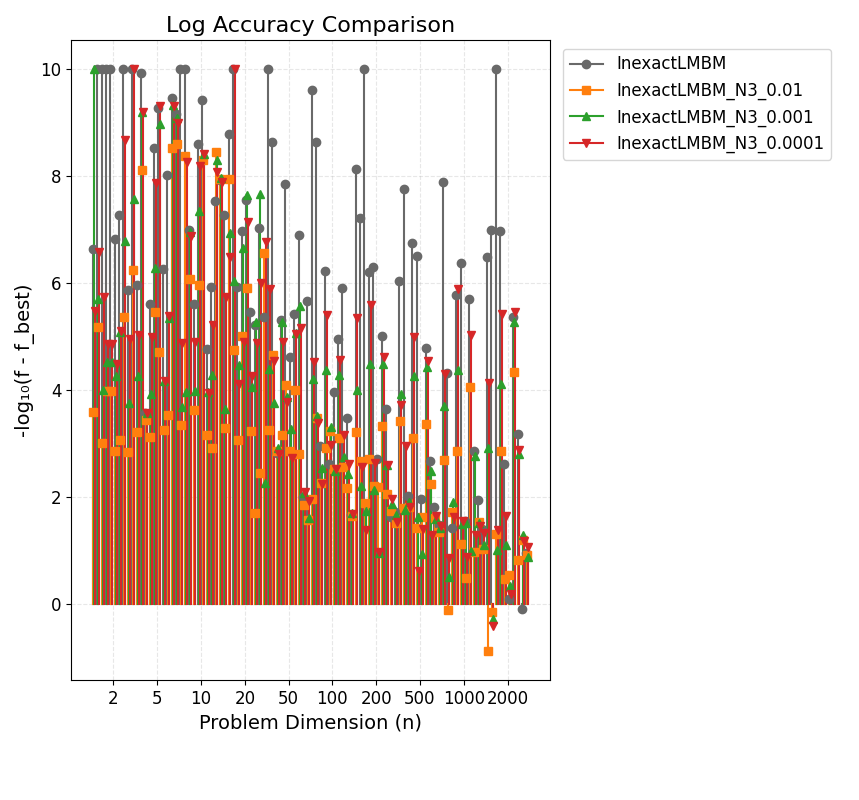}}%\hspace{-1.7cm}
\subfigure[{$\mathcal{N}4$}]{\label{fig:cpu_n2n4}
\includegraphics[width=0.49\textwidth]{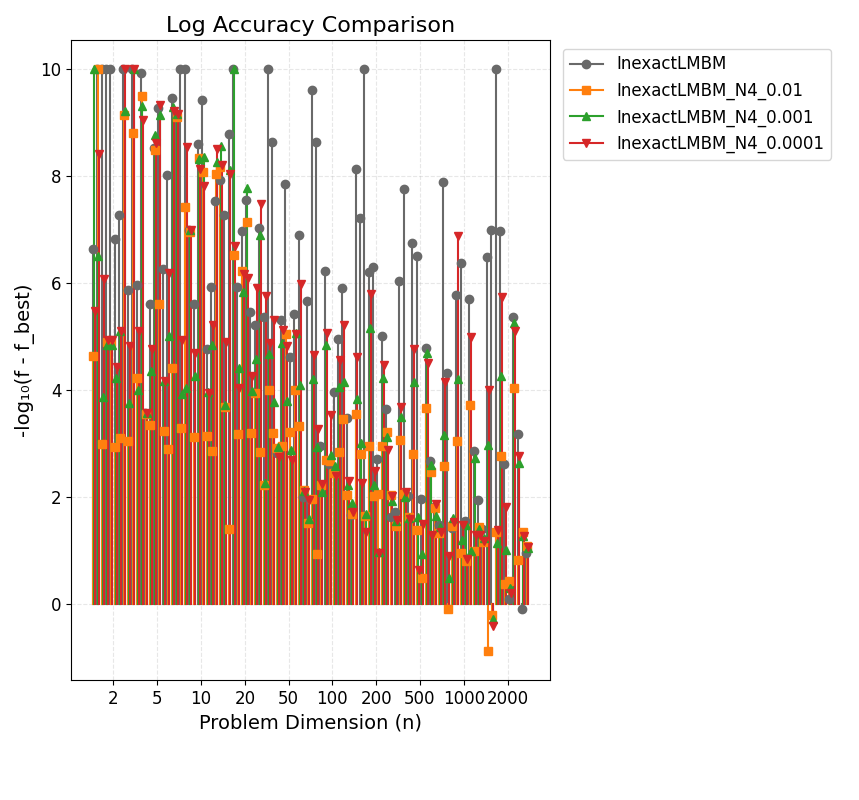}}%\hspace{-0.80cm}
    \caption{Comparison of different noise bounds  $\bar{q} = 0,\, 0.01, \, 0.001,$ and $0.0001$ with noise types (a) $\mathcal{N}1$ (constant noise in both function values and subgradients), (b) $\mathcal{N}2$ (vanishing noise in both function values and subgradients), (c) $\mathcal{N}3$ (constant noise in subgradients), and (d) $\mathcal{N}4$ (vanishing noise subgradients).
    }
    \label{fig:inexact_noise}
\end{figure*}

\clearpage

\section{Matrix updating}
\label{appendix_matrixupdating}

Here, we provide a more detailed description of the limited memory matrix-updating procedure for the matrix $D_k$ used in computing the search direction $\bd_k$. The procedure follows the original LMBM \cite{HaaMieMak:2004, HaaMieMak:2007}.

As described in the article, the main idea of the limited memory matrix updating is that, instead of explicitly storing the matrices $D_k$, information from a limited number of previous iterations is used to define $D_k$ implicitly. Thus, at each iteration, a small number of correction pairs $(\bs_i, \bu_i)$, $(i<k)$, are stored. As in the original LMBM, these correction pairs are used in the L-BFGS and L-SR1 updates to construct the matrix $D_k$. In particular, if the previous step was a serious step, the L-BFGS update is applied, whereas after a null step the L-SR1 update is used. The detailed formulas for the limited memory matrix updates are given next.

Let $\hat{m}_c$ denote the maximum number of stored correction pairs specified by the user ($\hat{m}_c \geq 3$), and define $\hat{m}_k = \min \{\,k-1, \hat{m}_c\,\}$ as the number of correction pairs currently stored. The $n \times \hat{m}_k$ matrices $S_k$ and $U_k$ are then formed as
$$
S_k = \begin{bmatrix}\bs_{k-\hat{m}_k} & \ldots &\bs_{k-1}\end{bmatrix}, \quad \text{and} \quad
  U_k = \begin{bmatrix}\bu_{k-\hat{m}_k} & \ldots & \bu_{k-1}\end{bmatrix}.
$$
When the storage capacity is reached, the oldest correction pairs are overwritten by the new pairs. Consequently, except for the first few iterations, the $\hat{m}_c$ most recent correction pairs $(\bs_i, \bu_i)$ are always available.

Following the approach in~\cite{ByrNocSch:1994}, the L-BFGS update is defined by
\begin{align}
  \label{eq_D_k_BFGS_in_LMBM}
  D_k = \vartheta_k I + \begin{bmatrix}S_k & \vartheta_k U_k\end{bmatrix}
  \begin{bmatrix}
    (R_k^{-1})^{\top} (C_k + \vartheta_k U_k^{\top} U_k) R_k^{-1} & -(R_k^{-1})^{\top} \\
    -R_k^{-1} & 0
  \end{bmatrix}
  \begin{bmatrix} S_k^{\top} \\ \vartheta_k U_k^{\top} \end{bmatrix}.
\end{align}
Here, $R_k$ is an $\hat{m}_k \times \hat{m}_k$ upper triangular matrix with entries
\begin{align*}
  (R_k)_{ij} = 
  \begin{cases}
    (\bs_{k-\hat{m}_k-1+i})^{\top} (\bu_{k-\hat{m}_k-1+j}), & i \le j, \\
    0, & \text{otherwise},
  \end{cases}
\end{align*}
$C_k$ is a $\hat{m}_k \times \hat{m}_k$ diagonal matrix given by
\begin{align*}
  C_k = \diag [ \bs_{k-\hat{m}_k}^{\top} \bu_{k-\hat{m}_k}, \ldots, \bs_{k-1}^{\top} \bu_{k-1} ],
\end{align*}
and the scaling parameter $\vartheta_k > 0$ is computed as
\begin{align}
\label{eq_vartheta_in_LMBM}
  \vartheta_k = \frac{\bu_{k-1}^{\top} \bs_{k-1}}{\bu_{k-1}^{\top} \bu_{k-1}}.
\end{align}
Furthermore, the L-SR1 update (see, e.g.,~\cite{ByrNocSch:1994}) is expressed as
\begin{align}
  \label{eq_D_k_SR1_in_LMBM}
  D_k = \vartheta_k I - (\vartheta_k U_k - S_k)
        (\vartheta_k U_k^{\top} U_k - R_k - R_k^{\top} + C_k)^{-1}
        (\vartheta_k U_k - S_k)^{\top},
\end{align}
where, unlike~(\ref{eq_vartheta_in_LMBM}), the scaling parameter is set to $\vartheta_k = 1$ for all $k$. For the exact algorithm for direction finding using the L-SR1 update \eqref{eq_D_k_SR1_in_LMBM}, see for example \cite{HaaMieMak:2007}.

\end{document}